\documentclass[11pt,reqno]{amsart}
\usepackage[utf8]{inputenc}
\usepackage{a4wide}
\usepackage{amsmath,mathrsfs,eucal,mathtools}
\usepackage{xcolor}
\usepackage{bbm,stix}
\usepackage{enumitem}
\usepackage{graphics} 
\usepackage{epsfig} 
\usepackage{graphicx,subfig}                     
\usepackage{wrapfig}
\usepackage{textcomp}
\usepackage{tikz}
\usepackage{pgfplots}
\usepackage{dsfont}
\usepackage{comment}
\usepackage{multicol}  
\usepackage{epstopdf}
\usepackage[colorlinks=true,allcolors=black]{hyperref}
\usepackage{bm}                        \usepackage[]{appendix}  
\usepackage{booktabs}                                                                          
\usepackage{graphicx,subfig}                                                       \usepackage{float} 
\usepackage{wrapfig}

\DeclarePairedDelimiter{\abs}{\lvert}{\rvert}

\DeclarePairedDelimiter{\norm}{\lVert}{\rVert}
\DeclarePairedDelimiter{\abss}{\bigg\lvert}{\bigg\rvert}

\DeclarePairedDelimiter{\prt}{(}{)}
\DeclarePairedDelimiter{\brk}{[}{]}

\theoremstyle{definition}
\newtheorem{rmk}{Remark}[section]
\newtheorem{definition}{Definition}[section]
\theoremstyle{plain}
\newtheorem{theorem}{Theorem}[section] 
\newtheorem{corollary}{Corollary}[section]
\newtheorem{lemma}{Lemma}[section]
\newtheorem{prop}{Proposition}[section] 

\renewcommand{\i}{_{i}}
\renewcommand{\j}{_{j}}
\renewcommand{\ij}{_{ij}}
\newcommand{\ii}{_{ii}}
\renewcommand{\k}{^{k}}
\newcommand{\ep}{^{\varepsilon}}
\newcommand{\epn}{^{\varepsilon,(n)}}
\newcommand{\de}{^{\delta}}
\newcommand{\ai}{{\alpha_i}}
\newcommand{\fori}{{i=1,\ldots, N}}

\newcommand{\R}{\mathbb{R}}
\newcommand{\N}{\mathbb{N}}
\newcommand{\W}{\mathcal{W}}
\newcommand{\T}{\mathcal{T}}
\newcommand{\ninf}{{L^\infty}}

\newcommand{\ro}{{\rho_1}}

\newcommand{\f}{\mathbf{f}}
\newcommand{\g}{\mathbf{g}}
\newcommand{\E}{\mathbf{E}}

\newcommand{\D}{\mathbf{D}}

\renewcommand{\ro}{\boldsymbol{\rho}}

\newcommand{\calE}{\mathcal E}
\newcommand{\calF}{\mathcal F}

\newcommand{\calM}{\mathcal M}
\newcommand{\calP}{\mathcal P}
\newcommand{\mc}{\mathcal C}
\newcommand{\e}{\varepsilon}
\newcommand{\intr}{\int_{\R^d}}

\newcommand{\intdd}{{\iint_{\R^d \times \R^d}}}
\newcommand{\Rdd}{{\R^d \times \R^d}}

\definecolor{green}{rgb}{0.01, 0.75, 0.24}


\begin{document}
\author{Young-Pil Choi \and Simone Fagioli \and Valeria Iorio}
\address{Young-Pil Choi - Department of Mathematics,
Yonsei University, 50 Yonsei-Ro, Seodaemun-Gu, Seoul 03722, Republic of Korea}
\email{ypchoi@yonsei.ac.kr}
\address{Simone Fagioli - DISIM - Department of Information Engineering, Computer Science and Mathematics, University of L'Aquila, Via Vetoio 1 (Coppito)
67100 L'Aquila (AQ) - Italy}
\email{simone.fagioli@univaq.it}
\address{ Valeria Iorio - DISIM - Department of Information Engineering, Computer Science and Mathematics, University of L'Aquila, Via Vetoio 1 (Coppito)
67100 L'Aquila (AQ) - Italy}
\email{valeria.iorio1@univaq.it}

\title[Small inertia limit for coupled kinetic swarming
models]{Small inertia limit for coupled kinetic swarming
models}
\date{}

\begin{abstract}
We investigate various versions of multi-dimensional systems involving many species, modeling aggregation phenomena through nonlocal interaction terms. We establish a rigorous connection between kinetic and macroscopic descriptions by considering the small-inertia limit at the kinetic level. The results are proven either under smoothness assumptions on all interaction kernels or under singular assumptions for \emph{self-interaction} potentials. Utilizing different techniques in the two cases, we demonstrate the existence of a solution to the kinetic system, provide uniform estimates with respect to the inertia parameter, and show convergence towards the corresponding macroscopic system as the inertia approaches zero.

\end{abstract}
\maketitle
  \tableofcontents
\section{Introduction}\label{sec:intro}
This paper aims to investigate the connections among different descriptions of multispecies aggregation phenomena, focusing specifically on nonlocal systems of partial differential equations. The main objective is to study the following first order macroscopic system
\begin{equation}
\label{eq:mainsystem} 
\begin{dcases}
\partial_t \rho\i + \nabla \cdot \prt{\rho\i v\i }=0, \\ 
v\i = - \sum_{j=1}^N \nabla K\ij \ast \rho\j,
\end{dcases}
\end{equation}
for $i=1,\ldots, N$, where $N$ is the number of species, $\rho\i (t,x)$ is a function modelling the $i$-th species density and $K_{ij}$ are given space-dependent interaction potentials modelling interaction between species. Interactions between agents of the same species are modelled by the $K_{ii}$ potentials that are called \textit{self-interaction} kernels whereas \textit{cross-interaction} kernels $K\ij$ describe the interactions of individuals of different species. 

System \eqref{eq:mainsystem} admits a discrete counterpart constructed as follows: consider $M$ particles for each species and let $z\i\k$, $k=1,\ldots,M$, be the locations of $M$ particles of the $i$-th species, for $\fori$. Denoting by $u\i\k$ the velocities of $z\i\k$, the dynamics of $z\i\k$ is determined by the first order ODE system
\begin{equation}
\label{eq:particlesystem}
\begin{dcases}
\frac{d z\i\k}{dt}=u\i\k, \\
u\i\k = -\frac{1}{M} \sum_{j=1}^N \sum_{h=1}^M \nabla K\ij \prt{z\i\k-z\j^h},
\end{dcases}
\end{equation}
for $\fori$, $k=1,\ldots,M$, where $K\ij$ are the same kernels as in \eqref{eq:mainsystem}. System \eqref{eq:particlesystem} can be also derived as \textit{small inertia limit} of the second order system
\begin{equation}
\label{eq:particlesystemeps}
\begin{dcases}
\frac{d}{dt}z\i\k = u\i\k, \\
\varepsilon \frac{d}{dt} u\i\k=-u\i\k-\frac{1}{M}\sum_{j=1}^N \sum_{h=1}^M \nabla K\ij \prt{ z\i\k-z\j^h} ,
\end{dcases}
\end{equation}
with $\varepsilon >0$, $i=1,\ldots , N$, $k=1,\ldots,M$, see \cite{bodnar}. 
In \eqref{eq:particlesystemeps}, $\varepsilon >0$ represents a small \textit{inertia} time of individuals. In system \eqref{eq:particlesystem}, it is assumed that the $\varepsilon$-terms in \eqref{eq:particlesystemeps} are negligible, but this choice is quite restrictive in many cases since in this way a ``reaction'' time is not taking into account and velocities change instantaneously.

Taking the formal limit as the number of particles increases to infinity, namely $M\to\infty$, we can associate to \eqref{eq:particlesystemeps} the kinetic system 
\begin{equation}
\label{eq:kinetic} 
\partial _t f\i+v\cdot \nabla _x f\i=\frac{1}{\varepsilon}\nabla_v\cdot (v f\i )+\frac{1}{\varepsilon}\nabla_v\cdot \bigg( \bigg(\sum_{j=1}^N \nabla K\ij \ast \rho\j \bigg) f\i \bigg),  
\end{equation}
for $\fori$, where $f\i(t,x,v)$ is the mesoscopic density of the $i$-th species at position $x\in \R^d$ with velocity $v\in\R^d$, and $\rho\i (t,x)$ is the associated macroscopic population density, i.e.,
\[
\rho\i(t,x)=\int_{\R^d} f\i(t,x,v)\,dv.
\]

The main goal of the present paper is to investigate the small inertia limit at the continuum level. In particular, we want to study the $\varepsilon\to 0$ limit in \eqref{eq:kinetic} and prove that it converges towards the first order PDEs model \eqref{eq:mainsystem}.

In recent decades, systems such as \eqref{eq:mainsystem} and \eqref{eq:kinetic} have been extensively employed to provide a biologically relevant representation of aggregative phenomena in population dynamics, particularly in the context of \emph{swarming} phenomena (see \cite{boi, mogilner, okubo, topaz}). Common interaction potentials in these scenarios include the attractive \emph{Morse} potential $G(x)=-e^{-|x|}$, attractive-repulsive Morse potentials $G(x)=-C_a e^{-|x|/l_a} + C_r e^{-|x|/l_r}$ (where $l_a$ and $l_r$ represent scales for the ``attractive range'' and the ``repulsive range'' respectively), combinations of Gaussian potentials $G(x)=-C_a e^{-|x|^2/l_a} + C_r e^{-|x|^2/l_r}$, or characteristic functions of a set $G(x)=\alpha \chi_{A}(x)$. Considerable attention is directed towards aggregation systems with singular kernels, particularly at the mesoscopic level. 

A notable mathematical characteristic of these models that has drawn attention is the finite-time  \emph{blow-up} of solutions. Numerous contributions have been made in the literature for the one-species version of \eqref{eq:mainsystem}, as seen in \cite{bertozzi1, bertozzi3, bertozzi2, bertozzi4, budif, CDFLS, choi_jeong_fractional, litoscani}. Notably, in \cite{choi_jeong_fractional} the one-species case of  \eqref{eq:mainsystem} with
\[
v=-\nabla K\ast \rho= c_K \Lambda^{\alpha-d}\nabla \rho,
\]
where $c_K\in\R$ and $-2\leq\alpha-d\leq0$ are parameters and $\Lambda^s$ is the $s-$fractional power of $\Lambda \coloneqq (-\Delta)^{\frac{1}{2}}$, to be defined precisely below, was studied  establishing the local-in-time existence and uniqueness of classical solutions.

Inspired by the results in the single-species case, an existence theory was developed for the system \eqref{eq:mainsystem} in \cite{difrafag}. In the case when the system presents a \emph{symmetry}, namely $K_{ij} = K_{ji}$ for $i\neq j$ and $i,j=1,\dots,N$, then system \eqref{eq:mainsystem} exhibits a formal gradient flow structure:
\[
\partial_t \rho_i = \nabla \cdot \bigg(\rho_i \nabla \frac{\delta \calF(\rho)}{\delta\rho_i}\bigg),
\]
for $\fori$, where $\calF(\rho)$ is a free energy given by
\[
\calF(\rho) = \sum_{i,j=1}^N \intr \rho_i K_{ij} *\rho_j\,dx.
\]
In \cite{difrafag}, the authors proved that the theory of Gradient Flows in Wasserstein spaces developed in \cite{ags, CDFLS} can be extended to systems under mildly singular assumptions on all kernels $K_{ij}$, i.e., Morse-type singularity. When the symmetry property is lost, the existence of weak-measure solutions is developed in \cite{difrafag} using a semi-implicit version of the \emph{JKO-scheme}, originally introduced in \cite{jko}. Existence can be proved under mildly singular assumptions on the self-interaction potentials and smoothness assumptions on the cross-interaction potentials. To the best of the authors' knowledge, no results in the literature cover the case of singular non-symmetric cross-interaction kernels or self-interaction kernels under more singular assumptions, similar to the ones in \cite{choi_jeong_fractional}.

Focusing on the system \eqref{eq:kinetic}, the kinetic approach is widely employed in investigating aggregation phenomena. In \cite{ccr}, the single-species version of equation \eqref{eq:kinetic} is examined, considering both a self-propulsion term and a friction term. The former influences individuals independently of others, while the latter introduces a velocity-averaging effect, compelling agents to adjust their velocities based on nearby agents. The paper presents results on well-posedness, existence, uniqueness, and continuous dependence in the space of probability measures $\mathcal{P}_1(\R^d)$ equipped with the Monge--Kantorovich--Rubinstein distance. Additionally, the corresponding microscopic system is explored, and a convergence result from the particle system to the kinetic equation is established. The existence of smooth solutions is addressed using the classical framework for Vlasov-type equations, as outlined in \cite{glassey}. We refer to \cite{iacobelli} and references therein for recent treatments of the Vlasov--Poisson equation. Let us also mention the contribution in \cite{carrillo_choi_peng} to the theory of the Vlasov--Poisson--Fokker--Planck system.

In \cite{fs}, equation \eqref{eq:kinetic} is further investigated in the one-species case in a multi-dimensional space, accounting for inertial effects. Assuming smooth conditions on the kernel and applying the theory developed in \cite{ccr}, the paper proves existence and uniqueness results in the sense of measures. A small inertia limit is also examined, providing convergence to the corresponding macroscopic system.

Starting from the seminal work \cite{kramers}, several contributions have appeared in the literature in the study of the limit from \eqref{eq:kinetic} to \eqref{eq:mainsystem} in its one species version; see \cite{freidlin,hott_volpe,narita1994asymptotic}. In \cite{Duong_2018,doung2017}, variational techniques were introduced to study the rigorous limit from the Vlasov--Fokker--Planck equation to the corresponding macroscopic equation under suitable regularity and integrability assumptions on the interaction potential. Finally, in \cite{Goudon2005, pousol00}, qualitative analysis of the \emph{overdamped} limit from the Vlasov--Poisson--Fokker--Planck system towards the drift-diffusion equation, with interaction kernels given as attractive or repulsive Coulomb potential, has been conducted. More recently, a rigorous \emph{quantified overdamped limits} for the Vlasov--Fokker--Planck equation with smooth and singular nonlocal forces have been established in \cite{RoSuXi21} and \cite{choi_tse}, respectively.

The rest of this paper is organised as follows. In Section \ref{sec:pre}, we present the collection of assumptions and the statements of the main theorems of the paper, namely the small inertia limits in the smooth case in Theorem \ref{th:kinetic_limit} and for singular self-interaction kernels in Theorem \ref{thm_ktoc}. Section \ref{sec:smooth} is devoted to the proof of Theorem \ref{th:kinetic_limit}. The existence of solutions to system \eqref{eq:kinetic} is first proved by applying the method of characteristic. Uniform in $\varepsilon$ estimates and convergence to solutions to \eqref{eq:mainsystem} is then proved by adapting to systems  the results in \cite{fs}. We then provide the proof of the main result of the paper that is Theorem \ref{thm_ktoc} in Section \ref{sec:singular_model}. A regularisation procedure yields the existence of solutions to system \eqref{eq:kinetic} in the case of singular self-interaction kernels. Suitable a-priori estimates allow us to prove the existence and uniqueness of classical solution to \eqref{eq:mainsystem} under the singular setting by extending the one-species result
in \cite{choi_jeong_fractional}. Finally Theorem \ref{thm_ktoc} is proved by using the modulated energy estimate technique. 

\section{Preliminaries and main results}\label{sec:pre}
\subsection{Preliminaries and assumptions}
\label{sec:kinetic_smooth_preliminaries}

Let $\mathcal{P}_1 (\R^d)$ be the set of probability measures with finite first moment, i.e.,
\[ 
\mathcal{P}_1(\R^d)=\bigg\{ f\in \mathcal{P}(\R^d) \, :\, \int_{\R^d} \abs{x}f(x)\,dx<\infty\bigg\}. 
\]
We equip $\mathcal{P}_1 (\R^d)$ with the $1$-Wasserstein distance defined as (cf.\ \cite{villani})
\[
W_1 (\mu, \nu)= \inf_{\gamma \in \Pi (\mu,\nu)} \left\{ \iint_{\R^d \times \R^d} \abs{x-y} \,d \gamma (x,y) \right\}
\]
for all $\mu, \ \nu \in \mathcal{P}_1 (\R^d)$, where $\gamma\in \mathcal{P}(\R^d\times\R^d)$ is a probability measure on the product space $\R^d \times \R^d$ with marginals $\mu$ and $\nu$, respectively. We also introduce the set of probability measures with compact support, that is
\[
\mathcal{P}_c(\R^d) = \bigg\{ f\in \mathcal{P}(\R^d) \, : \, f \, \mbox{has compact support} \bigg\}.
\]
Since we deal with $N$ interacting species, the measure space we consider is $( \mathcal{P}_1(\R^d)^N, \mathcal{W}_1)$, where $\mathcal{W}_1$ is the $1$-Wasserstein distance on $\mathcal{P}_1(\R^d)^N$ defined below. In order to fix the notation, we write
\[
\f=\prt{f\i}_{i=1}^N\in \mathcal{P}_1 \prt{\R^d}^N 
\]
to denote a $N$-tuple of probability measures in the product space $\mathcal{P}_1 \prt{\R^d}^N $.

\begin{definition}(1-Wasserstein distance) 
\label{def:W1}
Let $\f=\prt{f\i}_{i=1}^N, \, \g=\prt{g\i}_{i=1}^N\in \mathcal{P}_1 \prt{\R^d}^N$ . The \textit{$1$-Wasserstein distance between $\f$ and $\g$} is defined as
\[ \mathcal{W}_1(\f, \g)\coloneqq\sup_{t\in[0,T]}\big[ W_1(f_1,g_1)+ \cdots +W_1(f_N,g_N) \big]. \]
\end{definition}

Given $Q$ a generic potential involved in systems \eqref{eq:mainsystem} and \eqref{eq:kinetic}, we will call the potential \emph{smooth} if it satisfies the following assumption
\begin{equation}
\label{pot}
\tag{\textbf{Pot}}
Q \in \mc^2 (\R^d) \quad \mbox{and} \quad \nabla Q \in W^{1,\infty} (\R^d).  
\end{equation}  

\begin{rmk}[Lipschitz constant] Denote by $B_R$ a closed ball in $\R^d$ centered in $0$ and with radius $R>0$. We denote by $Lip_R(Q)$ the Lipschitz constant of $Q$ in the ball $B_R\subset\R^d$. If $Q$ depends also on time, i.e., $Q : [0,T]\times\R^d\to\R^n,$ $Q=Q(t,x)$, we write $Lip_R(Q)$ to denote the Lipschitz constant of $Q$ with respect to $x$ in the ball $B_R\subset\R^d$, that is the smallest constant such that
\[ \abs{Q(t,x)-Q(t,y)}\leq Lip_R(Q)\abs{x-y}, \]
for all $x,y\in B_R$, and for all $t\in [0,T]$.
\end{rmk}

    In the proposition below, we state a property on the convergence of measures, see \cite{villani} for more details.
\begin{prop}
\label{prop:convergence_properties} Let $(X,d)$ be a complete and separable metric space. Let $\mu_n$ be a sequence of probability measures in $\mathcal{P}_1(X)$, and let $\mu \in \mathcal{P}_1(X)$. Then, the following are equivalent:
\begin{itemize}
    \item[1.] $W_1 ( \mu_n, \mu)\to 0$ as $n\to \infty.$
    \item[2.] $\mu_n \to \mu$ in the weak sense as $n\to\infty$ and the following tightness condition holds: for any $x_0 \in X$, 
    \[ \lim_{R\to \infty} \limsup_{k \to \infty} \int_{d(x_0, x)\geq R} d(x_0, x) d\mu_n =0. \]
    \item[3.] $\mu_n \to \mu$ in the weak sense as $n\to\infty$, and there is a convergence of the moment of first order, i.e., for any $x_0 \in X$, 
    \[ \int_{X \times X} d(x_0,x) \,d\mu_n (x) \to \int_{X\times X} d(x_0,x) \,d\mu(x), \]
    as $n \to \infty$.
\end{itemize}
\end{prop}

\subsection{Main results}
In this subsection, we state our main results on the small inertia limits for the kinetic system \eqref{eq:kinetic} and its convergence to \eqref{eq:mainsystem}. We consider two cases: smooth and singular interactions potentials. In order to emphasize our results on the small inertia limit of the kinetic system \eqref{eq:kinetic} towards the first order macroscopic system \eqref{eq:mainsystem}, here we only state the theorems on that. The required existence theory for the systems \eqref{eq:kinetic} and \eqref{eq:mainsystem} will be discussed in later sections.

\subsubsection{Smooth potential case}
We start by introducing the notion of weak solutions to \eqref{eq:mainsystem}. We consider first the case of interaction potentials $K_{ij}$ under assumption \eqref{pot}.
\begin{definition}
\label{def:first_order_ws}
A weak solution to \eqref{eq:mainsystem} is a $N$-tuple $\ro = (\rho\i)_{i=1}^N \in \mc([0,T),\mathcal{P}(\R^d)^N)$ that satisfies 
\[
\int_0^T \int_{\R^d} \partial_t \phi\i \rho\i \,dx\,dt -\int_0^T \int_{\R^{d}} \nabla_x \phi\i \cdot \bigg( \sum_{j=1}^N K\ij \ast \rho\j \bigg) \rho\i \,dx\,dt +\int_{\R^d} \phi\i(0) \rho_{i0}\,dx=0,
\]
for each $\phi\i\in \mc_c^1([0,T);\mc_b^1(\R^{d}))$, as $\fori$. 
\end{definition}
We already mentioned in the previous section that the existence of weak solutions to \eqref{eq:mainsystem} can be found in \cite{difrafag}. For smooth interaction potentials the small inertia limit result reads as follows.
\begin{theorem}
\label{th:kinetic_limit} 
Let $T>0.$ Assume all the potentials as in \eqref{pot}. Consider $\f_0 \in\mathcal{P}_c(\R^{2d})^N$. Let $\f\ep \in \mc([0,T);\mathcal{P}_c(\R^{2d})^N)$ be the solution to system \eqref{eq:kinetic} given by Theorem \ref{th:existence}. Let $\rho\i\ep (t,x)=\int_\R f^\varepsilon\i (t,x,v)\,dv$, for $\fori$. Then there exists $\ro\in \mc([0,T);\mathcal{P}_1(\R^d)^N)$ such that for each $t\in [0,T)$,
\[
\ro\ep(t,\cdot)\xrightarrow{\mathcal{W}_1}\ro(t,\cdot) \qquad \text{in $\mathcal{P}_1(\R^d)^N$}
\]
as $\varepsilon\to 0$. Moreover, $\ro$ is a weak solution to system \eqref{eq:mainsystem} in the sense of Definition \ref{def:first_order_ws}.
\end{theorem}

Theorem \ref{th:kinetic_limit} will be proved in Section \ref{sec:smooth} by extending the results in \cite{fs} where the small inertia limit is proved in the one species case under regularity assumptions on the interaction kernel.

\subsubsection{Singular potential case}
Now we deal with the case of singular interaction potentials. Precisely, we consider the system \eqref{eq:kinetic} with smooth cross-potentials $K\ij$, $i\neq j$ satisfying \eqref{pot} and singular self-potentials $K_{ii}$ of the form
\begin{equation}
\label{eq:singular_potential0}
K\ii(x) \coloneqq \frac{C\i}{\abs{x}^\ai},
\end{equation}
with $\ai \in (0,d)$, and some positive constants $C\i$. Note that  if $\alpha_i \in ((d-2) \vee 0,d)$, then $K_{ii} * \rho_i = \Lambda^{\alpha_i-d}\rho_i$ with $\Lambda = (-\Delta)^{\frac12}$ up to constant. Thus in this case the system \eqref{eq:mainsystem} becomes the following coupled \textit{fractional porous medium flows} \cite{CSV13}:
\[
\partial_t \rho_i = \nabla \cdot \bigg(\rho_i \bigg(\nabla \Lambda^{\alpha_i-d}\rho_i + \sum_{ \substack{ j=1 \\ j\neq i}}^N \nabla K_{ij}\ast\rho_j\bigg)\bigg),
\]
for $\fori$. 

Then our second and main result is stated as follows.
\begin{theorem}\label{thm_ktoc} Let $T>0$ and $d \geq 1$. Let $\f^\e=(f^\e_i)_{i=1}^N \in \mc([0,T); \mathcal{P}(\R^d \times \R^d)^N)$ be a solution to system \eqref{eq:kinetic} in the sense of distributions, and let $(\ro, \mathbf{u})=(\rho_i, u_i)_{i=1}^N$ be the unique classical solution of the system \eqref{eq:mainsystem} with $\rho_i > 0$ on $\R^d \times [0,T)$, $\partial_t u_i + u_i \cdot \nabla u_i \in L^\infty(\R^d \times (0,T))$, and if $\alpha<d-2$, $\nabla^{[(d-\alpha)/2]+1} u_i\in L^\infty((0,T);L^{\frac{d}{[(d-\alpha)/2]}}(\R^d))$ up to time $T>0$ with the initial data $ \rho_{i0}$. If 
\begin{equation}\label{ini_cond_sing0}
\sup_{\e > 0}\sum_{i=1}^N\intdd \abs{v - u_{i0}(x) }^2 f^\e_{i0}(x,v)\,dx\,dv < \infty
\end{equation}
and
\begin{equation}\label{ini_cond_sing1}
\sum_{i=1}^N\intr (\rho_{i0} - \rho^\e_{i0})K_{ii}\ast(\rho_{i0} - \rho^\e_{i0})\,dx + \sum_{i=1}^N W_1(\rho_{i0}, \rho_{i0}^\e) \to 0,
\end{equation}
as $\e \to 0$, then for each $i=1,\dots, N$, we have
\begin{align*}
\intr \,f^\e_i\,dv &\xrightharpoonup{\ast} \rho_i \quad \mbox{weakly-$*$ in } L^\infty((0,T);\mathcal{M}(\R^d)), \cr
\intr v \,f^\e_i\,dv  &\xrightharpoonup{\ast} \rho_i u_i  \quad \mbox{weakly-$*$ in } L^2((0,T);\mathcal{M}(\R^d)), 
\end{align*}
and
\[
f^\e_i \xrightharpoonup{\ast} \rho_i\delta_{u_i} \quad \mbox{weakly-$*$ in } L^2((0,T);\mathcal{M}(\R^d \times \R^d)),
\]
where we denoted by $\mathcal{M}(\R^n)$ the space of signed Radon measures on $\R^n$ with $n \in \N$.
\end{theorem}

Our proof for Theorem \ref{thm_ktoc} relies on the modulated energy estimates. For this, we need to establish the existence theory for the kinetic system \eqref{eq:kinetic} and the first order macroscopic system \eqref{eq:mainsystem} at least locally in time. To be more specific, as stated in Theorem  \ref{thm_ktoc}, it suffices to construct the weak solutions to \eqref{eq:kinetic}, but for the limit system \eqref{eq:mainsystem}, it is required to show the existence and uniqueness of regular solutions satisfying the regularity conditions of Theorem \ref{thm_ktoc}. 

\begin{rmk}\label{rmk_sing_main} Here we provide some remarks regarding Theorem \ref{thm_ktoc}. 
\begin{itemize}
    \item[(i)] If we further assume 
\[
\sum_{i=1}^N\intdd \abs{v - u_{i0}(x) }^2 f^\e_{i0}(x,v)\,dx\,dv \to 0 
\]
and
\[
\frac1\e\sum_{i=1}^N\intr (\rho_{i0} - \rho^\e_{i0})K_{ii}\ast(\rho_{i0} - \rho^\e_{i0})\,dx + \frac1\e\sum_{i=1}^N W_1^2(\rho_{i0}, \rho_{i0}^\e) \to 0
\]
as $\e \to 0$, then for each $i=1,\dots,N$, we have
\[
 \intr f^\e_i\,dv \xrightharpoonup{\ast} \rho_i, \qquad  \intr v f^\e_i\,dv \xrightharpoonup{\ast} \rho_i u_i \quad \mbox{weakly-$*$ in } L^\infty((0,T); \calM(\R^d)),
 \]
 and
 \[
 f^\e_i \xrightharpoonup{\ast} \rho_i \delta_{u_i}  \quad \mbox{weakly-$*$ in } L^\infty((0,T); \calM(\R^d \times \R^d))
\]
as $\e \to 0$.

\item[(ii)] Due to some technical reasons, we are only able to construct the global-in-time $L^1\cap L^\infty$ solutions to the kinetic system \eqref{eq:kinetic} for $\alpha_i \in (0,\alpha_i  - 1]$, see Theorem \ref{th:existence_kin_singular}. We also need to assume additional condition $\nabla K_{ij} \in W^{1,1}(\R^d)$, $i\neq j$ to develop the local-in-time well-posedness of the macroscopic system \eqref{eq:mainsystem}, see Theorem \ref{thm_classical}. In that respect, our results are fully rigorous when the interaction potentials $K_{ij}$, $i\neq j$, satisfy
\[
K_{ij} \in C^2(\R^d), \quad \nabla K_{ij} \in W^{1,1}\cap W^{1,\infty}(\R^d)
\]
and $K_{ii}$ is given as \eqref{eq:singular_potential0} with $\alpha_i \in (0,d-1]$. 
\end{itemize} 
\end{rmk}

\section{Smooth interaction potentials}\label{sec:smooth}
\subsection{Well-posedness for the kinetic system for $\varepsilon>0$ fixed}
\label{sec:kinetic_smooth_wellpos}
We start the investigation of the small inertia limit for smooth interaction kernels by studying  the well-posedness for system \eqref{eq:kinetic} for $\varepsilon >0$ fixed, in the spirit of \cite{ccr}.
We start observing that such existence theory can be studied for a more general class of force fields
\[\E\coloneqq \prt{E\i}_{i=1}^N (t,x) : [0,T]\times \R^{d}\to \R^{Nd}, \]
for $\fori$, fulfilling the following general set of hypotheses:
\begin{itemize}
\item [$\mathbf{(H_1)}$] $E\i$ are continuous on $[0,T]\times \R^{d}$, for all $\fori$. 
\item [$\mathbf{(H_2)}$] There exist some positive constants $C\i$ such that
\[
\abs{E\i(t,x)}\leq C\i(1+\abs{x}),
\]
for all $(t,x)\in [0,T]\times \R^{d}$, for all $\fori$.
\item [$\mathbf{(H_3)}$] $E\i$ are locally Lipschitz with respect to $x$ uniformly in $t$, for $\fori$, that is for any compact set $K\subset \R^d$ there exist positive constants $L\i$ such that
\[
\abs{E\i (t,x)-E\i (t,y)}\leq L\i \abs{x-y}, 
\]
for all $x,y\in K$, and for all $t\in [0,T]$.
\end{itemize}

The kinetic system we are going to study takes then the following form
\begin{equation}
\label{eq:kineticgeneral}
\partial _t f\i+v\cdot \nabla _x f\i-\frac{1}{\varepsilon}\nabla_v\cdot (vf\i)+\frac{1}{\varepsilon}E\i \cdot \nabla_v f\i=0,
\end{equation}
for $\fori$. Following the approach in \cite{ccr}, we will construct solutions to \eqref{eq:kineticgeneral} by considering the following characteristic system associated to \eqref{eq:kineticgeneral}
\begin{equation}
\label{eq:charactgeneral}
\begin{dcases} 
\frac{dX}{dt}=V, \\ 
\frac{dV}{dt}=-\frac{1}{\varepsilon} V+\frac{1}{\varepsilon} E\i\prt{t,X},
\end{dcases}
\end{equation}
for $\fori$. Introducing $P\coloneqq \prt{X,V} \in \R^d\times \R^d$, and denoting by
\[
\Psi_{E\i} : [0,T]\times \R^d\times \R^d\to\R^d\times \R^d
\]
the right-hand side of system \eqref{eq:charactgeneral}, we can rewrite system \eqref{eq:charactgeneral} as
\begin{equation}
\label{eq:chatacter}
\frac{d}{dt} P=\Psi _{E\i} (t,P),
\end{equation}
for $\fori$, subject to the initial condition $P_0=(X_0, V_0)\in \R^d \times\R^d$. Existence and uniqueness of solutions to system \eqref{eq:chatacter} falls into the classical ordinary differential equations theory, see \cite{tikhonov}, that provides a vector field $P\in \mc^1([0,T];\R^{2d})$ such that
\[ \abs{P}\leq \abs{P_0}e^{Ct},\]
for all $t\in [0,T]$ and the constant $C$ depending on $T$, $\abs{X_0}$, $\abs{V_0}$. We will use the compact notation
\[ \Psi_\E \coloneqq \prt{\Psi_{E\i} }_{i=1}^N. \]



Let us introduce the time-dependent flow map associated to system \eqref{eq:charactgeneral} by 
\[ \mathcal{T}_{E\i}^t : \R^{d}\times\R^{d}\to\R^{d}\times\R^{d}, \]
such that
\[ \mathcal{T}_{E\i}^t \prt{ \prt{X_0, V_0}}=\prt{ X(t), V(t)}, \]
where $\prt{X(t),V(t)}$ is the unique solution to \eqref{eq:charactgeneral} at time $t>0$ under the initial condition $\prt{X_0, V_0}$.

We are now in the position to introduce the notion of measure solution to system \eqref{eq:kineticgeneral}. Consider $f_{i0} \in\mathcal{P}_1(\R^{2d})$ the initial datum of the $i$-th species and let $T>0$. Then, a measure solution to \eqref{eq:kineticgeneral} can be defined as
\[
f\i \prt{t} = \T_{E\i}^t \# f_{i0},
\]
for $\fori$. With a slight abuse of notation, for using a compact formulation we set
\[ \T_\E^t = \big( \T^t_{E\i} \big)_{i=1}^N, \] 
and given the initial datum $\f_0 \in \mathcal{P} \prt{\R^{2d}}^N$ and a time $T>0$, we define the measure solution to \eqref{eq:kineticgeneral} as
\[
\f(t)=\mathcal{T}_\E^t\# \f_0.
\]
Referring to system \eqref{eq:kinetic}, we define the vector field $\E\brk{ \f}$ associated to a $N$-tuple of measures $\f$ as
\begin{equation}
\label{eq:field_sol}
\E \brk{ \f} = \prt{E\i \brk{\f}}_{i=1}^N = \bigg( -\sum_{j=1}^N \nabla K\ij\ast	\rho\j \bigg)_{i=1}^N.
\end{equation}
We can now give the notion of measure solution to system \eqref{eq:kinetic} as in \cite{ccr,fs}.
\begin{definition} [Measure solution to \eqref{eq:kinetic}] \label{def:solution} Fix $T>0$ and $\varepsilon>0$. Let $\f_0\in\mathcal{P}_1(\R^{2d})^N$ be a given initial condition and let $\E\brk{\f}$ be defined as in \eqref{eq:field_sol}. A $N$-tuple $\f : [0,T]\to\mathcal{P}_1(\R^{2d})^N$ is a measure solution to system \eqref{eq:kinetic} with initial condition $\f_0$ if: 
\begin{itemize}
\item [1.] the field $\E\brk{\f}$ defined in \eqref{eq:field_sol} satisfies the conditions $\mathbf{(H_1)}$-$\mathbf{(H_2)}$-$\mathbf{(H_3)}$;
\item[2.] it holds $\f (t)=\T_{\E\brk{\f}}^t\# \f_0 $.
\end{itemize}
\end{definition}

\subsubsection{A priori estimates on the characteristics system} \label{subsec:characteristics}
In this part, we collect some results on the solution to the characteristic system \eqref{eq:charactgeneral}. Proofs of the two Lemmas below can be obtained directly from system \eqref{eq:charactgeneral} and by definition of $\Psi_\E$ in \eqref{eq:chatacter}, see \cite{ccr,fs}.

\begin{lemma}\label{lemma:regularity1} Fix $T>0$. Let $\E,\D : [0,T]\times \R^{d}\to \R^{Nd}$ be  fields that satisfy $\mathbf{(H_1)}$-$\mathbf{(H_2)}$-$\mathbf{(H_3)}$ and let $\Psi_\E, \Psi_\D$ as in \eqref{eq:chatacter}. Consider $R>0$ and the closed ball $B_R\subset \R^{Nd}\times \R^{Nd}$. Then
\begin{itemize}
\item [1.] $\Psi_\E$ is bounded in compact sets, i.e.,
\[ 
\abs{\Psi_\E (t,P)}\leq C,
\]
for all $P\in B_R$, $t\in [0,T]$ and for some $C>0$ which depends on $R$ and $\norm{\E}_{L^\infty ([0,T]\times B_R^1)},$ where $B_R^1$ is the ball in $\R^d$ with radius $R$.
\item [2.] $\Psi_\E$ is locally Lipschitz with respect to $X$ and $V$, i.e.,
\[ \abs{\Psi_\E (t,P_1)-\Psi_\E(t,P_2)}\leq C (1+Lip_R(\E))\abs{P_1-P_2}, \]
for all $P_1, P_2\in B_R$, $t\in [0,T]$ and $C>0.$
\item [3.] For any compact set $B\subset \R^{Nd}\times\R^{Nd}$,
\[ \norm{\Psi_\E-\Psi_\D}_{L^\infty (B)}\leq \frac{1}{\varepsilon} \norm{\E-\D}_{L^\infty (B^1)}. \]
\end{itemize}
\end{lemma}


Now we provide some results that concern the dependence of the characteristics on the field $\E$ and a quantitative bound on the regularity of the flow $\mathcal{T}_\E ^t$.

\begin{lemma}\label{lemma:depending2} Fix $T>0$ and consider  two vector fields $\E$ and $\D$ satisfying $\mathbf{(H_1)}$-$\mathbf{(H_2)}$-$\mathbf{(H_3)}$. Take $P_0,P_1,P_2\in\R^{Nd}\times\R^{Nd}$ and $R>0$. Assume
\[ \abs{\mathcal{T}^t_\E(P_0)}\leq R, \quad \abs{\mathcal{T}^t_\D(P_0)}\leq R,\quad \abs{\mathcal{T}^t_\E(P_1)}\leq R, \quad \abs{\mathcal{T}^t_\E(P_2)}\leq R, \]
for $t\in [0,T]$. Then,
\begin{itemize}
    \item [1.] there is a constant $C$ depending on $R$ and $Lip_R(\E)$ such that
\[ \abs{\mathcal{T}^t_\E(P_0)-\mathcal{T}^t_\D(P_0)}\leq\frac{e^{Ct}-1}{C\varepsilon}\sup_{s\in[0,T)} \norm{\E(s)-\D(s)}_{L^\infty (B_R^1)}, \]
for $t\in [0,T]$.
\item [2.] There is a constant $C$ depending on $R$ such that
\[ \abs{\mathcal{T}^t_\E(P_1)-\mathcal{T}^t_\E(P_2)}\leq \abs{P_1-P_2} e^{C\int_0^t (Lip_R(\E(s))+1)\,ds}, \]
for $t\in [0,T]$.
\item[3.] There is a constant $C$ depending on $R$ and  $\norm{\E}_{L^\infty ([0,T]\times B_R^1)}$ such that
\[ \abs{\mathcal{T}^t_\E(P_0)-\mathcal{T}^s_\E(P_0)}\leq C\abs{t-s}, \]
for $s,t\in [0,T]$. 
\end{itemize}
\end{lemma}

\begin{rmk}
Note that the sub-linearity assumption on the vector field $\E$ ensures global existence for solution for $t\in\R$. The boundedness assumption in Lemma \ref{lemma:depending2} on the initial flow $\mathcal{T}^t_{\E} (P_0)$ is only needed to prove a quantitative estimate on the flow map for every time $t \in [0,T]$. Moreover, Lemma \ref{lemma:depending2} ensures that the flow $\mathcal{T}^t_\E$ is Lipschitz on $B_R\subset \R^{Nd}\times\R^{Nd}$, with constant
\[ Lip_R(\mathcal{T}^t_\E)\leq e^{C\int_0^t (Lip_R(\E(s))+1)\,ds}, \]
for $t\in [0,T]$.

\end{rmk}

In the following lemmas, we collect some contraction results in the Wasserstein distance $\mathcal{W}_1$ that are crucial in proving the existence of measure solutions for \eqref{eq:kineticgeneral}. What we reproduce is the extension to multiple species of the results in  \cite[Lemmas 3.11, 3.12, and 3.13]{ccr}, see also \cite{fs}. 

\begin{lemma} \label{lemma:bou_dist_init_data}
Let $\E, \,\D : \R^{Nd}\to\R^{Nd}$ be two Borel measurable maps and let $\f\in\mathcal{P}_1(\R^d)^N.$ Then
\[ \W_1\big( \E\#\f,\D\#\f \big) \leq \norm{\E-\D}_{\ninf(\emph{supp}\, \f)}. \]
\end{lemma}

\begin{lemma}\label{lemma:cont_time} Let $T>0.$ Let $\E: [0,T]\times\R^{d}\to\R^{Nd}$ be a field that satisfies $\mathbf{(H_1)}$-$\mathbf{(H_2)}$-$\mathbf{(H_3)}$ and let $\f$ be a $N$-tuple of measures on $\R^d$ with compact support contained in a ball $B_R\subset \R^d$. Then, there exists a positive constant $C$ depending on $N$, $R$ and $\norm{\E}_{L^\infty ([0,T]\times B_R^1)}$ such that
\[ \mathcal{W}_1 (\mathcal{T}_\E^t\#\f, \mathcal{T}_\E^s\#\f)\leq C\abs{t-s}, \]
for any $s,t\in [0,T]$.
\end{lemma}

\begin{lemma} \label{lemma:lip_contr}
Let $\mathcal{T} : \R^{Nd}\to\R^{Nd}$ be a Lipschitz map and let $\f,\g\in\mathcal{P}_1(\R^d)^N$ both have compact support contained in a ball $B_R.$ Then
\[ \W_1 \big( \mathcal{T}\#\f, \mathcal{T}\#\g \big) \leq L \W_1\prt{ \f,\g },\]
where $L$ is the Lipschitz constant of $\T$ on the ball $B_R$.
\end{lemma}

\subsubsection{Existence and uniqueness for smooth potentials}
We turn now to the existence and uniqueness of measure solutions to system \eqref{eq:kinetic}. We first provide the following preliminary Lemmas, whose proof is straightforward and we omit.

\begin{lemma} \label{lemma:norm_bound}
Assume the potentials $K\ij$ under assumption \eqref{pot}. Let $\f\in\mathcal{P}_1(\R^{2d})^N$ be with compact support contained in a ball $B_R\subset \R^{2d}$. Set $B_R^1\coloneqq \{ x\, : \, \prt{x,v}\in B_R \}$. Consider the vector field defined in \eqref{eq:field_sol}. Then,
\[ \norm{\E[\f]}_{\ninf(B_R^1)}\leq \Xi, \quad\mbox{ and }
\quad Lip_R(\E[\f])\leq \Upsilon, \]
where the constants $\Xi$ and $\Upsilon$ are defined by
\[
\Xi \coloneqq \sum_{i,j=1}^N \norm{\nabla K\ij}_{\ninf(B_{2R})},
\]
and
\[
\Upsilon\coloneqq \sum_{i,j=1}^N Lip_{2R}(\nabla K\ij).\]
\end{lemma}

\begin{lemma} \label{lemma:norm_bound_2} Assume the potentials $K\ij$ as in \eqref{pot}. Let $\f,\g \in\mathcal{P}_1(\R^{2d})^N$ and $R>0.$ Then,
\[ \norm{\E[\f]-\E[\g]}_{\ninf (B_R^1)} \leq \Upsilon \W_1\prt{  \f,\g }. \]
\end{lemma}

Existence and uniqueness of measure solutions to the kinetic system \eqref{eq:kinetic} is stated and proved in the following Theorem.

\begin{theorem} \label{th:existence} Assume the potentials $K\ij$ under assumption \eqref{pot}, and let $\f_0 \in \mathcal{P}_c(\R^{2d})^N$. Then there exists a unique measure solution $\f \in \mathcal{P}_c(\R^{2d})^N$ to system \eqref{eq:kinetic} with initial condition $\f_0$ in the sense of Definition \ref{def:solution}. In particular, 
\begin{equation}
\label{eq:solutioninC} 
\f \in \mc([0,+\infty);\mathcal{P}_c(\R^{2d})^N),
\end{equation}
and there exists an increasing function $R=R(T)$ such that for all $T>0$,
\begin{equation}
\label{eq:solutionbounded} 
\emph{supp}(\f)\subset B_{R(T)}\subset\R^d \times \R^d,
\end{equation}
for all $t\in[0,T]$.
\end{theorem}
\begin{proof}
Let $\f_0$ be such that 
\[ \text{supp}\prt{ \f_0} \subset B_{R_0}\subset\R^d \times \R^d, \]
for some $R_0>0.$ In order to prove the existence and uniqueness of the solution, we are going to use a contraction argument. In particular, we introduce the metric space
\[
\mathcal{F} = \left\{ \f \in \mc((0,T],\mathcal{P}_c(\R^{2d})^N) \,  :  \, \mbox{supp}\prt{ \f} \subset B_R \mbox{ for all }\,t\in[0,T]\right\} ,
\]
where $R\coloneqq 2R_0$ and $T>0$ is a fixed time we will choose later. This metric space is equipped with the distance $\mathcal{W}_1$, see Definition \ref{def:W1}.  In this space we define a map as follows. For $\f \in\mathcal{F}$, consider $\E \brk{\f}$ defined as in \eqref{eq:field_sol}. Then, by Lemmas \ref{lemma:norm_bound} and \ref{lemma:norm_bound_2} and by assumption \eqref{pot}, we obtain that $\E \brk{ \f}$ satisfies $\mathbf{(H_1)}$-$\mathbf{(H_2)}$-$\mathbf{(H_3)}$ and thus we can define
\[ \Gamma [\f](t) \coloneqq \mathcal{T}^t_{\E[\f]}\# \f_0. \]
The aim is to prove that this map is a contraction and its unique fixed point in $\mathcal{F}$ is the solution to \eqref{eq:kinetic}. 
 We start proving that the operator $\Gamma[\f]$ is well-posed in the space $\mathcal{F}$. From Lemma \ref{lemma:norm_bound} we have that
\[ \norm{\E[\f]}_{L^\infty([0,T]\times B_R^1)} \leq \Xi, \]
and from Lemma \ref{lemma:regularity1},
\[ \abss{\frac{d}{dt} \mathcal{T}^t_{\E[\f]} (P)}\leq C_1, \]
for all $P\in B_{R_0}\subset \R^d\times\R^d$, with $C_1$ depending on $R_0$ and $\Xi$. For $T<R_0/C_1$, we have that $\mathcal{T}^t_{\E[\f]}\#\f_0$ has support contained in $B_R$ for all $t\in [0,T]$. Then, for each $t\in [0,T]$, $\Gamma[\f](t)\in\mathcal{P}_c(\R^{2d})^N$ and the map $t\mapsto \Gamma[\f](t)$ is continuous by Lemma \ref{lemma:cont_time}. Thus the map $\Gamma:\mathcal{F}\to\mathcal{F}$ is well defined. \\
We show now that the map is a contraction, i.e., considering two functions $\f, \g\in\mathcal{F}$ and taking $\Gamma[\f]$ and $\Gamma[\g]$, we want to prove that
\[\mathcal{W}_1(\Gamma[\f],\Gamma[\g])\leq C\mathcal{W}_1(\f,\g) \]
for $0<C<1$ which does not depend on the functions $\f$ and $\g$. By definition of $\Gamma$ we have that
    \[ \mathcal{W}_1(\Gamma[\f],\Gamma[\g]) = \W_1 \prt{ \T_{\E[\f]}^t\#\f_0,\T_{\E[\g]}^t\#\f_0 }. \]
Using Lemmas \ref{lemma:bou_dist_init_data}, \ref{lemma:depending2} and \ref{lemma:norm_bound_2}, the above distance can be estimated as follows
\begin{align*}
\W_1 \prt{ \T_{\E \brk{\f}}^t \# \f_0, \T_{\E \brk{\g}}^t \# \f_0 } & \leq \norm{\T_{\E\brk{\f}}^t -\T_{\E\brk{\g}}^t } _{L^\infty \prt{\text{supp}\, \f_0}} \\&
\leq C(t) \sup_{s\in [0,T]} \norm{\E [\f] (s) - \E [\g] (s) }_{L^\infty \prt{B_R^1}} \\ &
\leq C(t) \Upsilon \W_1 \prt{\f,\g},
\end{align*}
where $C(t)=(e^{C_2t}-1)/\varepsilon C_2$ is the function in the statement of Lemma \ref{lemma:depending2}, with $C_2$ a constant depending on $R$ and $\Upsilon$.
Therefore, we obtain that
\[ 
\mathcal{W}_1 (\Gamma [\f],\Gamma[\g]) \leq C(t) \Upsilon \mathcal{W}_1(\f,\g).
\]
Since it holds that
\[ \lim_{t\to 0} C(t)=0, \]
we get
\[ \W_1 (\Gamma [\f],\Gamma [\g])\leq C(T)\Upsilon\W_1(\f,\g). \]
We can choose $T$ small enough so that $C(T)\Upsilon<1$. In this way, the functional $\Gamma$ is contractive and then there is a unique fixed point of $\Gamma$ in $\mathcal{F}$. By construction it is easy to see that this fixed point of $\Gamma$ is a solution to \eqref{eq:kinetic} on $[0,T]$. Finally, since the growth of characteristic is bounded we can construct a unique global solution satisfying \eqref{eq:solutioninC} and \eqref{eq:solutionbounded}.
\end{proof}

We now state $\W_1-$stability for solutions to system \eqref{eq:kinetic} that can be proved by using Lemmas \ref{lemma:bou_dist_init_data}, \ref{lemma:depending2} and \ref{lemma:norm_bound_2} and Gr\"onwall's lemma. 
\begin{prop} \label{prop:stability} Assume that the potentials $K\ij$ are under assumption \eqref{pot}. Let $\f_0, \g_0 \in\mathcal{P}_c(\R^{2d})^N$, and consider the solutions $\f,\g$ to \eqref{eq:kinetic} with initial conditions $\f_0$ and $\g_0$, respectively. Then, there exists an increasing function $r(t):[0,\infty)\to \R^+$ with $r(0)=1$ that depends only on the supports of $\f_0$ and $\g_0$ such that 
\begin{equation} \label{eq:stability}
\mathcal{W}_1(\f(t),\g(t))\leq r(t) \mathcal{W}_1(\f_0, \g_0),
\end{equation}
for $t\geq 0$.
\end{prop}

\begin{theorem} \label{th:smoothexistence} Let $T>0$ be a positive time. Assume $K\ij$ as in \eqref{pot}. Let $\f_0\in \mc^1(\R^{2d})^N \cap L^1(\R^{2d})^N$ with compact support. Then system \eqref{eq:kinetic} has a solution $\f^\varepsilon\in \mc([0,T];\mc^1(\R^{2d})^N)$ with initial datum $\f_0$.
\end{theorem}
\begin{proof}[Sketch of Proof]
We provide a sketch of the proof. The details can be found in \cite{glassey} for the Vlasov-Poisson system and the Vlasov-Maxwell system.
The proof is divided into three steps. One first constructs an approximating sequence $\f^{\e,n} \in\mc ([0,T]; \mc^1 (\R^d \times \R^d)^N)$ by iterations, defining $\f^{\e,n+1}$ to be the solution of
\begin{gather*}
\partial_t f\i^{\e,n+1}  + v \cdot \nabla_x f\i^{\e,n+1} - \frac{1}{\varepsilon} \nabla_v \cdot (v f\i^{\e,n+1} ) + \frac{1}{\varepsilon} E\i^{n} \cdot \nabla_v f\i^{\e,n+1} =0, \\
f\i^{\e,n+1} (0,x,v)= f_{i0}(x,v),
\end{gather*}
as $\fori$.
Once we observe that the characteristics associated to the system above depend on $n$, but still satisfy the features in Lemmas above, then it holds that $\f^{\e,n} \in \mc^1 ( [0,T]; \mc^1 (\R^d \times\R^d)^N)$ is uniformly bounded with respect to $n$, since all the constants in the estimates above depend on the support of the initial datum and  the Lipschitz constant of the kernels. Next, showing that $\f^{\e,n}$ is a Cauchy sequence in $\mc ([0,T]; \mc^1 (\R^d \times \R^d)^N)$ converging to the solution to \eqref{eq:kinetic}, we complete the proof.
\end{proof}

\subsection{Uniform estimates in $\varepsilon$}
\label{sec:kinetic_smooth_estim}
In this part, we gather some uniform in $\varepsilon$ estimates we will use to prove the convergence of solutions to \eqref{eq:kinetic} towards the solution to \eqref{eq:mainsystem} as $\varepsilon\to 0$. For this reason, we make the $\varepsilon$-dependence explicit, i.e., we deal with the system
\begin{equation}
\label{eq:system_uniform}
\partial _t f\i^\varepsilon+v\cdot \nabla _x f\i^\varepsilon=\frac{1}{\varepsilon}\nabla_v\cdot \bigg( \bigg( v+\sum_{j=1}^N \nabla K\ij \ast \rho_j^\varepsilon \bigg) f\i^\varepsilon \bigg), \end{equation}
for $\fori$, equipped with initial data
\begin{equation}
\label{initial_unif}
f\i^\varepsilon \prt{ t,x,v} \vert_{t=0} =f_{i0}^\e(x,v) \in \mathcal{P}_1(\Rdd), 
\end{equation}
and where
\[
\rho\i^\varepsilon(t,x)=\int_{\R^d} f\i^\varepsilon(t,x,v)\,dv.
\]
Throughout this Section, we assume the initial data with compact support.

\begin{prop} \label{prop:support} Assume all the potentials under assumption \eqref{pot}. Let $\f^\varepsilon$ be a solution to the system \eqref{eq:system_uniform}-\eqref{initial_unif} as proved in Theorem \ref{th:existence}. Then, there exists an increasing function $R(T)$ independent on $\varepsilon$ such that for all $T>0$,
\[
\emph{supp}(\f^\varepsilon)(t)\subset B_{R(T)},
\]
for all $t\in [0,T]$ and $\varepsilon >0$.
The function $R(T)$ depends only on the support of $\f_0^\e$ and $\Xi$.
\end{prop}
The proof easily follows by noticing that the support of $\f^\varepsilon$ evolves according to the flow associated to the characteristic system \eqref{eq:charactgeneral} and by the bounds
\[ \abss{ \sum_{j=1}^N \nabla K\ij \ast \rho_j ^\varepsilon } \leq \sum_{j=1}^N \norm{\nabla K\ij}_{L^\infty} \eqqcolon C_{i,1}, \]
for $\fori$, as in \cite{fs}.

\subsubsection{Estimate for smooth solutions} We first produce uniform in $\varepsilon$ estimates in the case of smooth solutions, namely solutions given by Theorem \ref{th:smoothexistence}. In the next subsection, we will deal with uniform in $\e$ estimates for measure solutions.

\begin{prop}\label{prop:estimate_smooth} Assume all the potentials under assumption \eqref{pot}. Suppose that the initial datum $\f_0$ in \eqref{initial_unif} has a finite first moment in $v$, i.e., $\abs{v}f_{i0} \in L^1(\R^d\times \R^d)$ for all $\fori$. Let $\f^\varepsilon$ be the classical solution to \eqref{eq:system_uniform}, as in Theorem \ref{th:smoothexistence}.  Then there exist some positive constants $C\i$, as $\fori$, and a function $M(\varepsilon)$ depending on $\varepsilon$, such that
\begin{equation}
\iint_{\R^{2d}} \abss{v+ \sum_{j=1}^N \nabla K\ij \ast \rho_j ^\varepsilon } f\i^\varepsilon \,dx\,dv \leq C\i M(\varepsilon), \label{eq:mainsmoothrho} 
\end{equation}
for all $t\in [0,T]$, where $C\i$ depends on $\norm{(1+\abs{v})f_{i0}}_{L^1(\R^{2d})}$ and $\Xi$. Moreover,
\[ \lim_{\varepsilon\downarrow0}M(\varepsilon)=0.\]
\end{prop}
\begin{proof}
We set
\[
I\i(t)=\iint_{\R^{2d}} \abss{v+\sum_{j=1}^N \nabla K\ij \ast \rho_j^\varepsilon }f\i^\varepsilon\,dx\,dv,
\]
as $\fori$. We want to prove that there exist $C\i$ and $M(\varepsilon)$ as in the statement such that
\[
\sup_{t\in [0,T]} I\i(t)\leq C\i M (\varepsilon),
\]
for a small $\varepsilon.$
Straightforward computation shows that
\begin{align*}
\frac{d}{dt} I\i (t) = & \iint_{\R^d\times\R^d} \bigg( \partial_t \abss{v+\sum_{j=1}^N \nabla K\ij\ast\rho\j^\varepsilon} \bigg) f\i^\varepsilon \,dx\,dv \\
& + \iint_{\R^d\times\R^d} \abss{v+\sum_{j=1}^N \nabla K\ij\ast \rho\j^\varepsilon} \partial_t f\i^\varepsilon\,dx\,dv.
\end{align*}
By using system \eqref{eq:system_uniform} and integration by parts, we get
\begin{align*}
\iint_{\R^d\times\R^d} \abss{v+\sum_{j=1}^N \nabla K\ij\ast \rho\j^\varepsilon} \partial_t f\i^\varepsilon\,dx\,dv 
= & \iint_{\R^d\times \R^d} \bigg( v \cdot \nabla_x \abss{v+\sum_{j=1}^N \nabla K\ij\ast \rho\j^\varepsilon} \bigg) f\i^\varepsilon\,dx\,dv \\
& \quad - \frac{1}{\varepsilon} \iint_{\R^d \times\R^d} \abss{v+\sum_{j=1}^N \nabla K\ij\ast \rho\j^\varepsilon} f\i^\varepsilon\,dx\,dv.
\end{align*}
Therefore, we have that
\[
\frac{d}{dt}I\i(t)=-\frac{1}{\varepsilon} I\i(t)+I\i^1(t)+I\i^2(t), \\
\]
with
\begin{align*}
I\i^1(t)&=\iint_{\R^{2d}}\bigg( \partial_t \abss{v+ \sum_{j=1}^N \nabla  K\ij\ast \rho_j^\varepsilon } \bigg) f\i^\varepsilon \,dx\,dv,  \\ 
I\i^2(t)&=\iint_{\R^{2d}} \bigg( v\cdot \nabla_x \abss{v+ \sum_{j=1}^N \nabla K\ij \ast \rho_j^\varepsilon} \bigg) f\i^\varepsilon \,dx\,dv,
\end{align*}
for $\fori$. In order to obtain our claim, we want to show that $I\i^1(t)$ and $I\i^2(t)$ are bounded linearly by $I\i(t)$, and then derive a differential inequality to bound $I\i(t)$.
By setting
\[
    m\i \coloneqq \int_{\R^d} v f\i\,dv,
\]
integrating \eqref{eq:system_uniform} in $v$, we have that
\begin{equation}
\label{eq:cons_mass}
\partial _t \rho\i^\varepsilon+\nabla _x\cdot  m\i^\varepsilon =0,
\end{equation}
and the  conservation of masses
\[
\norm{\rho\i^\varepsilon(t)}_{L^1(\R^d)}=\norm{f_{i0}}_{L^1(\R^{2d})}, \]
for all $t>0$ and for all $\fori$. Thus, using the equation for $\rho\i^\varepsilon$ in \eqref{eq:cons_mass}, we can preliminary estimate
\[
\abss{ \partial_t\abss{v+ \sum_{j=1}^N \nabla K\ij \ast \rho_j^\varepsilon } } \leq \abss{ \sum_{j=1}^N \nabla K\ij \ast \partial _t \rho_j^\varepsilon }\leq   \sum_{j=1}^N \abss{\Delta K\ij \ast m_j^\varepsilon}.
\]
By adding and subtracting $\sum_{h=1}^N \nabla K_{ih} \ast \rho_h\ep$ in the absolute value in the right hand side of the inequality above, and using assumption \eqref{pot} we get
\begin{align*}
\sum_{j=1}^N \abss{\nabla K\ij \ast m\j\ep } \leq & \sum_{j=1}^N \norm{ \Delta K\ij}_\ninf \int_{\R^d} \abss{ w + \sum_{h=1}^N \nabla K_{ih}\ast \rho_h\ep} f\j\ep \prt{x,w,t}\,dw \\&
\quad + \sum_{j=1}^N \norm{\Delta K\ij}_\ninf \norm{\nabla K\ij}_\ninf \norm{\rho\j\ep}^2_{L^1}.
\end{align*}
Thus, integrating the above inequality in $x$ and $v$ we obtain
\begin{align*}
\abs{I\i^1 \prt{t}} 
\leq & \sum_{j=1}^N \norm{\Delta K\ij}_\ninf I\i\prt{t} \norm{\rho\j\ep}_{L^1} \\ &
\quad + \sum_{j=1}^N \norm{ \Delta K\ij}_\ninf \bigg( \int_{\R^d} \int_{\R^d} \abs{v} f\i\ep \prt{x,v,t} \,dx\,dv + \int_{\R^d}\int_{\R^d} \abs{w} f\j\ep \prt{x,w,t}\,dx,dw \bigg) \\&
\quad + \sum_{j=1}^N \norm{ \Delta K\ij}_\ninf \norm{ \nabla K\ij}_\ninf \norm{\rho\j\ep}^2_{L^1} \norm{\rho\i\ep}_{L^1}.
\end{align*}
Since $f\i\ep \in \mathcal{P}_1 \prt{\R^{2d}}$ for all $\fori$, we obtain that for each $i$ there exist two positive constants $A\i^1$ and $A\i^2$ depending on $\Xi$ and all $\norm{\prt{1+\abs{v}} f_{i0}}_{L^1}$ such that
\[ I\i^1(t) \leq A\i^1 I\i(t)+A\i^2. \]
Concerning the terms $I\i^2$, we can estimate
\begin{align*}
\abs{I\i^2 (t)} \leq & \sum_{j=1}^N \int_{\R^d\times\R^d} \abs{v} \abs{\Delta K\ij \ast \rho\j\ep} f\i\ep\,dx\,dv \\ 
\leq & \sum_{j=1}^N \norm{ \Delta K\ij}_\ninf \norm{\rho\j\ep}_{L^1} \bigg[ \iint_{\R^d \times \R^d} \abss{ v+\sum_{h=1}^N \nabla K_{ih} \rho_h\ep} f\i\ep \,dx\,dv \\ & \hspace{4cm} + \sum_{h=1}^N \iint_{\R^d \times \R^d} \abss{ \nabla K_{ih} \ast \rho_h\ep} f\i\ep \,dx \,dv \bigg] \\
\leq & \sum_{j=1}^N \norm{ \Delta K\ij} _\ninf \norm{\rho\j\ep}_{L^1} \bigg[ I\i(t) + \sum_{h=1}^N \norm{\nabla K_{ih} } _\ninf \norm{ \rho_h\ep}_{L^1} \norm{ \rho\i\ep}_{L^1} \bigg].
\end{align*}
Thus, we derive that for each $i$ there exist two positive constants $B\i^1$ and $B\i^2$ depending on $\Xi$ and all $\norm{ f_{i0}}_{L^1}$ such that
\[ \abs{I\i^2(t)} \leq B\i^1 I\i(t) + B\i^2. \]
Hence, considering the estimates above, we obtain that
\begin{equation}
\label{eq:mainproof1}
\frac{d}{dt} I\i (t) \leq -\frac{1}{\varepsilon} I\i(t)+C\i^1 I\i(t)+C\i^2, 
\end{equation} 
where $C\i\k = A\i\k+B\i\k$ for $\fori$ and $k=1,2$. Furthermore, at time $t=0$ we get
\begin{equation}
\label{eq:mainproof2}
I\i(0) \leq \norm{ \abs{v} f_{i0}}_{L^1} + D\i, 
\end{equation}
as $\fori$, where the positive constants $D\i$ depend on $\Xi$ and all $\norm{ f_{i0}}_{L^1}$.
Combining \eqref{eq:mainproof1} and \eqref{eq:mainproof2} and using Gr\"onwall's lemma, we obtain that
\[
\sup _{t \in [0,T]} I\i(t) \leq C\i M\i (\varepsilon),
\] 
where the constants $C\i$ depend on $\Xi$ and all $\norm{(1+\abs{v})f_{i0}}_{L^1}$. Finally, is it enough to note that $M\prt{\varepsilon} \coloneqq \max\i \{ M\i \prt{\varepsilon}  \}$ decays to $0$ as $\varepsilon
\to 0$.
\end{proof}

\subsubsection{Estimate for measure solutions} \label{subsect:estimate_measure_valued}
In this section, our aim is to find an estimate as in \eqref{eq:mainsmoothrho} for a measure solution $\f$ to system \eqref{eq:system_uniform}. In order to proceed, we introduce the mollifier 
\[ \gamma^{(n)}(x,v)=n^{2d}\gamma^{(1)}(nx,nv)\in \mc_c^\infty(\R^{2d}), \]
where
\begin{gather*}
\text{supp}(\gamma^{(1)})\subset \overline{B(0,1)}\subset \R^{2d}, \qquad \gamma^{(1)}\geq 0, \qquad \iint_{\R^{2d}} \gamma^{(1)} (x,v)\,dx\,dv = 1, \\ \iint_{\R^d\times\R^d} \abs{v}\gamma^{(1)} (x,v)\,dx\,dv\leq 1.
\end{gather*}

Now, let $\f_0\in \mathcal{P}_c(\R^{2d})^N$ and $\varepsilon >0$ fixed. Define
\begin{equation}
\label{eq:moll_rho}
\f_0^{(n)}=\f_0\ast\gamma^{(n)},
\end{equation}
i.e., 
\[ f_{i0}^{(n)}=f_{i0}\ast\gamma^{(n)}\in \mc^2(\R^{2d}), \]
 for all $\fori$. 
The following is a classical result concerning the mollifier $\gamma$, see \cite{ambrosio03}.

\begin{lemma} \label{lemma:moll_properties} Let $\f\in\mathcal{P}_1(\R^{2d})^N$ be with $\emph{supp}(\f)\subset B(R_0)\subset \R^{2d}.$ Then
\begin{itemize}
\item[(i)] $\emph{supp}\prt{\f^{(n)}}\subset B(R_0+1)$ for all
$n\geq 1$.
\item[(ii)] $\f^{(n)}\in \mathcal{P}_1(\R^{2d})^N$ and 
\[
\iint_{\Rdd} \abs{v}f\i^{(n)}(x,v)\,dx\,dv
\]
are uniformly bounded, for all $\fori$.
\item[(iii)] $\{ \f^{(n)}\}_{n\geq 1}$ is a Cauchy sequence in $\mathcal{P}_1(\R^{2d})^N$ equipped with the Wasserstein distance $\mathcal{W}_1$ and $\norm{\f^{(n)}-\f}_{\mathcal{W}_1}\to 0$ as $n\to +\infty$.
\end{itemize}
\end{lemma}

Consider that the approximating sequence $\f\epn$ satisfies the system
\begin{equation}
\label{eq:system_moll}
\partial _t f\i\epn+v\cdot \nabla _x f\i\epn=\frac{1}{\varepsilon}\nabla_v\cdot \bigg( \bigg( v+ \sum_{j=1}^N \nabla K\ij \ast \rho\i\epn \bigg) f\i\epn \bigg),
\end{equation}
for $\fori$, equipped with initial data
\[
f\i\epn\mid_{t=0}=f^{(n)}_{i0}(x,v), 
\]
with 
\[ \rho\i\epn=\int_{\R^d} f\i\epn\,dv. \]
 The results in Theorem \ref{th:smoothexistence}, Proposition \ref{prop:support} and Proposition \ref{prop:stability}, together with Lemma \ref{lemma:moll_properties},  ensure the following approximation result. 
\begin{lemma} \label{lemma:convergence}
Assume all the potentials under assumption \eqref{pot}. Let $\f_0\in\mathcal{P}_c(\R^{2d})^N$ and $\f_0^{(n)}$ as in \eqref{eq:moll_rho}. Then for each $T>0$ there exists a  solution $\f\epn\in \mc([0,T); \mc^1(\R^{2d})^N)$ to \eqref{eq:system_moll} whose support depends only on $T$ and $\Xi$ and is uniformly bounded both in $\varepsilon$ and $n$. Furthermore, if $\f\ep\in \mc \prt{ [0,T),\mathcal{P}_c\prt{\R^{2d}}^N}$ is the unique measure solution to \eqref{eq:system_uniform} as provided in Theorem \ref{th:existence}, then
\[
\f\epn \prt{t,\cdot,\cdot} \xrightarrow{\mathcal{W}_1}\f\ep \prt{t,\cdot,\cdot} \quad \text{in $\mathcal{P}_1(\R^{2d})^N$,}
\]
uniformly in $t$ as $n\to\infty$.
\end{lemma}

From this result it follows that
\begin{equation}
\label{eq:convergence_measure}
\ro\epn(t,\cdot)\to \ro\ep(t,\cdot) \quad \text{weakly as measures} 
\end{equation}
for each $t\in [0,T)$ as $n\to\infty$, where $\ro = \prt{	\rho\i}_{i=1}^N$.

\begin{lemma} \label{lemma:convolution_convergence}
 Let $\f\ep$ be the solution to \eqref{eq:system_uniform} obtained as the limit of approximating sequences $\f\epn$ as in Lemma \ref{lemma:convergence}. Then, for all $t\geq 0$, 
$ \nabla K\ij\ast\rho\j\ep$ are continuous functions in $\R^d$ for all $i,j=1\ldots,N$
and
\[
\nabla K\ij\ast\rho\j\epn(t,\cdot) \to\nabla K\ij\ast\rho\j\ep(t,\cdot) 
\]
strongly in $L^\infty_{loc}(\R^d)$, as $n\to\infty.$
\end{lemma}
\begin{proof}
Given the regularity of $\f\epn$, i.e., $\f\epn\in \mc([0,T); \mc^1(\R^{2d})^2)$ and the assumption \eqref{pot}, we get the continuity of the convolutions. Moreover, we can easily estimate
\[
\abs{ \nabla K\ij\ast\rho\j\epn } \leq \norm{\nabla K\ij}_\ninf, 
\]
and for all $x_1,x_2\in\R^d$, we have
\[
\abs{\nabla K\ij\ast\rho\j\epn(x_1)-\nabla K\ij\ast\rho\j\epn(x_2)}\leq \norm{\nabla K\ij}_\ninf \abs{x_1-x_2}, 
\]
for $i,j=1,\ldots,N$. Thus the sequences 
\[ 
\{\nabla K\ij\ast\rho\j\epn\}_{n\geq 1}
\]
are equicontinuous and uniformly bounded. Hence, by Ascoli-Arzelà theorem, they strongly converge on a subsequence on compact sets in $\R^d$. Furthermore, by \eqref{eq:convergence_measure} we have that the limit functions are 
\[
\nabla K\ij\ast\rho\j\ep, 
\]
respectively. These limit functions are also continuous on $\R^d$ by inequalities above (using $\rho\j\ep$ in place of $\rho\j\epn$). Then it follows the assertion.
\end{proof}


Since the approximating sequence $\f\epn$ is smooth, we can apply to it Proposition \ref{prop:estimate_smooth} with $\varepsilon$ fixed. In particular, with $n\geq 1$ fixed, we can say that there exist $N$ positive constants $C\i$ depending on $\Xi$ and all $\norm{(1+\abs{v}f_{i0}^{(n)})}_{L^1}$ and a function $M\prt{\varepsilon}$ depending on $\varepsilon$ such that for $\varepsilon<\varepsilon_0$
\[
\abss{\iint_{\R^{2d}} \bigg( v+ \sum_{j=1}^N \nabla K\ij\ast\rho\j\epn \bigg) f\i\epn \,dx\,dv}\leq C\i M(\varepsilon).
\]
By part $(ii)$ in Lemma \ref{lemma:moll_properties}, we have that $\norm{(1+\abs{v}f_{i0}^{(n)})}_{L^1}$ are uniform bound in $n$ for all $\fori$, thus the function $M(\varepsilon)$ and the constants $C\i$ can be chosen independent on $n$. Therefore the estimates
\begin{equation}
\label{eq:estimate_moll_rho}
\abss{\iint_{\R^{2d}} \bigg( v+\sum_{j=1}^N \nabla K\ij \ast \rho\j\epn \bigg) f\i\epn \,dx\,dv}\leq C\i M(\varepsilon)
\end{equation}
hold for all $n\geq 1$ and $t\in [0,T]$, as $\fori$.

\begin{prop} \label{prop:estimate_moll} Assume $\varepsilon >0$ fixed such that \eqref{eq:estimate_moll_rho} holds and assume that assumptions in Lemma \ref{lemma:convergence} are satisfied. Then for any $\prt{\phi\i}_{i=1}^N \in \mc_b(\R^{2d})^N$ there exist $N$ constants $\overline{C}\i$ such that
\[
\abss{\iint_{\R^{2d}}\phi\i(x,v) \bigg( v+\sum_{j=1}^N \nabla K\ij \ast \rho\j\ep (x) \bigg) f\i\ep(x,v)\,dx\,dv}\leq \overline{C\i} M(\varepsilon)
\]
hold for all $t\in [0,T]$, as $\fori$.
In particular, the constants $\overline{C}\i$ are independent of $\varepsilon$ and $t$, and $\overline{C}\i=\norm{\phi\i}_\ninf C\i$, where $C\i$ are constants depending on all $\iint (1+\abs{v})f_{i0}\,dx\,dv$ and $\Xi$.
\end{prop}
\begin{proof}
Multiplying \eqref{eq:estimate_moll_rho} by $\phi\i$ we have
\begin{equation}
\label{eq:measure_valued_rho}
\abss{\iint_{\R^{2d}} \phi(x,v) \bigg( v+\sum_{j=1}^N \nabla K\ij \ast \rho\j\epn \bigg) f\i\epn \,dx\,dv}\leq C\i\norm{\phi\i}_\ninf M(\varepsilon), 
\end{equation}
where $C\i$ are constants depending on $\Xi$ and the first moment of $f_{i0}$ in $v$. Let $\Omega (T)$ be the common support of $\f\epn (t)$ for all $\varepsilon >0,$ $n\geq 1$ and $t\in [0,T].$ Then, by Lemma \ref{lemma:convolution_convergence} and Proposition \ref{prop:convergence_properties}, we obtain that for each $t\in [0,T],$
\begin{align*}
\iint_{\R^{2d}} & \phi \i (x,v) \bigg( v+\sum_{j=1}^N \nabla K\ij \ast \rho\j\epn \bigg) f\i\epn \,dx dv 
\\ = & \iint_{\Omega (T)} \phi\i (x,v) \bigg( v+\sum_{j=1}^N \nabla K\ij \ast \rho\j\epn \bigg) f\i\epn \,dx dv 
\end{align*}
converges to 
\[
\iint _{\R^{2d}} \phi\i(x,v) \bigg( v+\sum_{j=1}^N \nabla K\ij \ast \rho\j\ep \bigg) f\i\ep \,dx dv
\]
as $n\to\infty$, for all $\fori$. Therefore, considering the limit as $n\to\infty$ in \eqref{eq:measure_valued_rho}, we find the assertion.
\end{proof}

\subsection{Small inertia limit}
\label{sec:kinetic_smooth_limit}
This subsection is finally  devoted to the proof of Theorem \ref{th:kinetic_limit}. More precisely, we consider $\f\ep$ solution to \eqref{eq:kinetic}, satisfying the uniform bounds as stated in Proposition \ref{prop:estimate_moll} and we show that the marginals
\[
\rho\i\ep(t,x)=\int_{\R^d} f\i\ep(t,x,v)\,dv
\]
converge to a solution $\ro = (\rho\i)_{i=1}^N$ to the first order system 
\begin{equation} 
\label{eq:system_compact}
\partial_t\rho\i -\nabla\cdot \bigg( \bigg( \sum_{j=1}^N \nabla K\ij \ast \rho\j \bigg) \rho\i \bigg) =0,
\end{equation}
for $\fori$, equipped with initial data
\[
\rho\i (t,x) \mid_{t=0}=\rho_{i0}(x).
\]


\begin{proof}[Proof of Theorem \ref{th:kinetic_limit}]
We start noting that, for each $\phi\i\in \mc_c^1([0,T);\mc_b^1(\R^{2d}))$, the measure solution $\f\ep$  satisfies  
\begin{equation}
\label{eq:weak_solution_fg}
\begin{aligned}
\int_0^T \iint_{\R^{2d}} \partial_t\phi\i f\i\ep\,dx\,d&v\,dt +\iint_{\R^{2d}} \phi\i(0) f_{i0}\,dx\,dv + \int_0^T \iint_{\R^{2d}} \nabla_x \phi\i \cdot vf\i\ep\,dx\,dv\,dt  \\ & -\frac{1}{\varepsilon}\int_0^T \iint_{\R^{2d}} \nabla_v \phi\i \cdot \bigg(v+\sum_{j=1}^N \nabla K\ij \ast \rho\j\ep\bigg)f\i\ep\,dx\,dv\,dt =0,
\end{aligned}
\end{equation}
for all $\fori$. Consider $\psi\i \in \mc_c^1(0,T)$, and $\chi\i \in \mc_b^1(\R^d)$, as $\fori$, and define 
\begin{equation}
\label{test_functions}
    \phi\i(t,x,v)=\psi\i(t)\chi\i(x).
\end{equation}
Using the test functions  defined in \eqref{test_functions} in system \eqref{eq:weak_solution_fg}, we have
\[
\int_0^T \psi\i'(t)\int_{\R^d} \chi\i(x)\rho\i\ep (t,x) \,dx\,dt=-\int_0^T\psi\i(t)\iint_{\R^{2d}} \nabla_x\chi\i(x)\cdot vf\i\ep \,dx\,dv\,dt.
\]
Set
\[ 
\xi\i(t)\coloneqq \int_{\R^d} \chi\i(x)\rho\i\ep (t,x)\,dx.
\]
Thus, it follows
\[ 
\int_0^T \psi\i'(t)\xi\i(t)\,dt=-\int_0^T \psi\i(t)\iint_{\R^{2d}}\nabla_x\chi\i(x)\cdot vf\i\ep\,dx\,dv\,dt, 
\]
for any $\psi\i \in \mc_c^1(0,T).$ Therefore, we deduce that the weak derivative of $\xi\i$ is
\[ 
\xi\i'(t)=\iint_{\R^{2d}} \nabla_x\chi\i\cdot vf\i\ep \,dx\,dv \in \ninf (0,T).
\]
Let $\Omega (T)$ be the common support of $\f\ep$ for every $\varepsilon >0$ and for all $t\in [0,T].$ By Theorem \ref{th:existence}, $\f\ep$ is uniformly supported on $\Omega (T)$, thus
\begin{equation} 
\label{eq:kinetic_th_bound_rho}
\norm{\xi\i}_{W^{1,\infty}(0,T)} \leq C\i (T)\norm{\chi\i}_{\mc_b^1(\R^d)},  
\end{equation}
where $C\i$ depend on $T$ and are independent of $\varepsilon.$
Since $\xi\i(t)$ are uniformly bounded in $W^{1,\infty}(0,T)$, as $\fori$, by Ascoli-Arzelà theorem there exist a subsequence $\varepsilon_k$ and a function $\mu\i(t)\in \mc([0,T))$ such that
\begin{equation}
\label{eq:kinetic_th_mu_rho} 
    \int_{\R^d} \chi\i(x)\rho\i^{\varepsilon_k} (x,t)\,dx \to \mu\i(t)
\end{equation}
uniformly on $[0,T)$ as $\varepsilon_k\to 0.$ Furthermore, Proposition \ref{prop:support} ensures that the support of $\f\ep$ is uniformly bounded in $\varepsilon$, then the sequence $\ro\ep(t,\cdot)$ is tight. By Prokhorov's Theorem, for each $t\in [0,T)$, $\ro\ep(t,\cdot)$ converges weakly-$^\ast$, up to a subsequence, to $\ro(t,\cdot)\in\mathcal{P}(\R^d)^N.$
By Proposition \ref{prop:convergence_properties}, we have that this implies convergence in $\mathcal{P}_1(\R^d)^N$ with respect to $\mathcal{W}_1$-distance. Hence, for each $t>0$, there exists a subsequence of $\ro^{\varepsilon_k}$ denoted by $\ro^{\varepsilon_{k_n}}$, where $k_n$ may depend on time, such that
\[ 
\ro^{\varepsilon_{k_n}}(t,\cdot)\xrightarrow{\mathcal{W}_1}\ro(t,\cdot) \quad \text{in $\mathcal{P}_1(\R^d)^N$} 
\]
as $\varepsilon_{k_n}\to 0$. It follows that for each $t\in [0,T)$ and all $\chi\i \in \mc_b^1(\R^d)$ we get
\begin{equation}
\label{eq:kinetic_th_weak_rho} 
\int_{\R^d} \rho\i^{\varepsilon_{k_n}}(t,x)\chi\i(x) \,dx \to \int_{\R^d} \rho\i(t,x)\chi\i(x)\,dx 
\end{equation}
as $\varepsilon_{k_n}\to 0.$
The limit $\mu\i(t)$ in \eqref{eq:kinetic_th_mu_rho} is unique at each $t\in[0,T).$ Combining this with \eqref{eq:kinetic_th_weak_rho}, we deduce that the sequence $\ro_{\varepsilon_k}(t,\cdot)$, with $\varepsilon_k$ independent of time, and $\ro(t,\cdot)\in\mathcal{P}_1(\R^d)^N$ satisfy
\begin{equation}
\label{eq:kinetic_th_lim_rho}
    \int_{\R^d} \chi\i(x) \rho\i^{\varepsilon_k}(t,x)\,dx \to \int_{\R^d} \chi\i(x)\rho\i(t,x)\,dx
\end{equation}
uniformly on $[0,T)$ as $\varepsilon_k\to 0$, for any $\chi\i\in \mc_b^1(\R^d)$. Moreover,
\begin{equation}
\label{eq:kinetic_th_vanish_meas}
\ro^{\varepsilon_k}(t,\cdot)\xrightarrow{\mathcal{W}_1} \ro(t,\cdot) \quad \text{in $\mathcal{P}_1(\R^d)^N$}
\end{equation}
as $\varepsilon_k\to 0$. Now we want to prove that in \eqref{eq:kinetic_th_lim_rho} we can consider test functions $\chi\i$ depending also on $t$. In particular, taking $\zeta\i(t,x)\in \mc_c([0,T); \mc_b^1(\R^d))$ we have that
\[ \int_{\R^d} \zeta\i(t,x)\rho\i^{\varepsilon_k}(t,x)\,dx \]
are equicontinuous on $[0,T)$.
Indeed, considering $s,t\in [0,T),$
\begin{align*}
  &  \abss{\int_{\R^d} \zeta\i(t,x)\rho\i^{\varepsilon_k}(t,x)\,dx-\int_{\R^d} \zeta\i(s,x)\rho\i^{\varepsilon_k}(s,x)\,dx} \\ & \quad \leq
  \int_{\R^d} \abs{\zeta\i(t,x)-\zeta\i(s,x)}\rho\i^{\varepsilon_k}(t,x)\,dx + \abss{\int_{\R^d} \zeta\i(s,x) [\rho\i^{\varepsilon_k}(t,x)-\rho\i^{\varepsilon_k}(s,x)]\,dx} \\ & \quad \leq \sup_{x\in\R^d} \abs{\zeta\i(t,x)-\zeta\i(s,x)}+C\i(T)\sup_{t\in (0,T)} \norm{\zeta\i}_{\mc_b^1(\R^d)} \abs{t-s}.
\end{align*}
Since $\zeta\i$ is uniformly continuous on $[0,T)\times\R^d,$ then
\[ \sup_{x\in\R^d} \abs{\zeta\i(t,x)-\zeta\i(s,x)}\to 0 \quad \text{as $\abs{t-s}\to 0$,} \]
and we get equicontinuity. Thus, up to a subsequence,
\begin{equation}
\label{eq:kinetic_th_limit_rho}
\int_{\R^d} \zeta\i(t,x)\rho\i^{\varepsilon_k}(t,x)\,dx\to\int_{\R^d} \zeta\i(t,x)\rho\i(t,x)\,dx 
\end{equation}
uniformly on $[0,T)$ as $\varepsilon_k\to 0$, for any test functions $\zeta\i \in \mc_c([0,T);\mc_b^1(\R^d))$. \\
Now, set
\[ \Omega_1(T)\coloneqq \{ x\, :\, (x,v)\in\Omega(T)\}. \]
We can deduce that $\Omega_1$ is bounded and both $\text{supp}\prt{\ro}$ and $\text{supp}\prt{\ro^{\varepsilon_k} }$ are in $\Omega_1(T)$ for all $t\in [0,T].$
Consider $\Psi\i \in \mc_c^1([0,T);\mc_b^1(\R^d))$ and let $\phi\i(x,v,t)=\Psi\i(x,t)$ in \eqref{eq:weak_solution_fg}, as $\fori$. Hence
\begin{equation}
\label{eq:weak_rho}
\int_0^T\int_{\R^d} \partial_t\Psi\i \rho\i^{\varepsilon_k}\,dx\,dt+\int_0^T\iint_{\R^{2d}} \nabla_x\Psi\i \cdot vf\i^{\varepsilon_k}\,dx\,dv\,dt +\int_{\R^d}\Psi\i(0)\rho_{i0} (x)\,dx=0.
\end{equation}
Regarding the first integral in \eqref{eq:weak_rho}, by \eqref{eq:kinetic_th_limit_rho} we have that
\[
    \int_0^T\int_{\R^d} \partial_t \Psi\i\rho\i^{\varepsilon_k} \,dx\,dt\to \int_0^T\int_{\R^d} \partial_t \Psi\i \rho\i\,dx\,dt, 
    \]
as $\varepsilon_k\to 0.$
Concerning the integrand of the second term in \eqref{eq:weak_rho}, it can be rewritten as
\begin{align*} 
\iint_{\R^{2d}} \nabla_x\Psi\i\cdot vf\i^{\varepsilon_k}\,dx\,dv= & \iint_{\R^{2d}} \nabla_x \Psi\i \cdot \bigg(v+\sum_{j=1}^N \nabla K\ij \ast \rho\j^{\varepsilon_k} \bigg) f\i^{\varepsilon_k} \,dx\,dv \\& \quad -\iint_{\R^{2d}} \nabla_x\Psi\i  \cdot \bigg( \sum_{j=1}^N \nabla K\ij\ast\rho\j^{\varepsilon _k}\bigg) f\i^{\varepsilon_k}\,dx\,dv.
\end{align*}
By Proposition \ref{prop:estimate_moll}, we have that
\begin{equation}
\label{eq:first_integral_rho}
    \iint_{\R^{2d}} \nabla_x \Psi\i \cdot \bigg(v+\sum_{j=1}^N \nabla K\ij \ast \rho\j^{\varepsilon_k} \bigg) f\i^{\varepsilon_k} \,dx\,dv\to 0, 
\end{equation}
as $\varepsilon_k\to 0$, uniformly in $t$.
The families $\{ \nabla K\ij\ast\rho\j^{\varepsilon_k}(t,\cdot)\}$ are bounded in $W^{1,\infty}(\R^d)$ for all $t\in [0,T)$. In particular,
\[
 \norm{\nabla K\ij\ast\rho\j^{\varepsilon_k}(t,\cdot)} _{W^{1,\infty}(\R^d)}\leq \norm{\nabla K\ij}_{W^{1,\infty}(\R^d)}.
 \]
Now, we want to prove that $ \{ \nabla K\ij\ast\rho\j^{\varepsilon_k} \}$ are equicontinuous in $t$.
In order to use inequalities in \eqref{eq:kinetic_th_bound_rho} with the kernels in places of $\chi\i$, we should mollify $K\ij.$ Let 
\[
    K\ij ^{(n)}=K\ij\ast\gamma^{(n)}, 
\]
where $\gamma^{(n)}$ is the mollifier defined in Section \ref{subsect:estimate_measure_valued}. 
It follows that
\[
    \nabla K\ij ^{(n)}=\nabla  K\ij\ast\gamma^{(n)}, 
\]
thus, we have
\begin{gather*}
    \norm{\nabla K\ij^{(n)}}_{\mc_b^1}\leq \norm{\nabla K\ij}_{W^{1,\infty}},
\end{gather*}
for all $n\geq 1.$ Now, considering the mollified interaction kernels acting on the $i$-th species in estimates \eqref{eq:kinetic_th_bound_rho} in places of $\chi\i$, we get
\[
    \sup_x\norm{ \sum_{j=1}^N \nabla K\ij^{(n)}\ast\rho\j^{\varepsilon_k}}_{W^{1,\infty}(0,T)} \leq C\i(T) \sum_{j=1}^N \norm{\nabla K\ij^{(n)}}_{\mc_b^1}  \leq C\i(T) \sum_{j=1}^N \norm{\nabla K\ij}_{W^{1,\infty}} .
\]
Furthermore, 
\begin{align*}
    \bigg\lVert\sum_{j=1}^N \nabla K\ij^{(n)}\ast\rho\j^{\varepsilon_k}\bigg\rVert_{W^{1,\infty}(\R^d)} \leq \sum_{j=1}^N \norm{\nabla K\ij^{(n)}}_{W^{1,\infty}}  \leq \sum_{j=1}^N \norm{\nabla K\ij}_{W^{1,\infty}},
\end{align*}
thus, we find that
\[
    \bigg \lVert \sum_{j=1}^N \nabla K\ij^{(n)}\ast\rho\j^{\varepsilon_k} \bigg\rVert_{W^{1,\infty}(\R^d \times (0,T))} \leq (1+C\i(T)) \sum_{j=1}^N \norm{\nabla K\ij}_{W^{1,\infty}}.
\]
Therefore, for all $x,y\in\R^d$ and $s,t\in [0,T)$, we get
\begin{equation}
\label{eq:kinetic_th_uniflim_rho}
\begin{aligned}
    \abss{ \sum_{j=1}^N \big[ & \nabla K\ij^{(n)}\ast\rho\j^{\varepsilon_k}(t,x) -\nabla K\ij^{(n)}\ast\rho\j^{\varepsilon_k}(s,y) \big] } \\ & \leq (C\i(T)+1) \bigg( \sum_{j=1}^N  \norm{\nabla K\ij}_{W^{1,\infty}} \bigg) \big( \abs{t-s}+\abs{x-y}\big).
\end{aligned}
\end{equation}
Since $\nabla K\ij$ are continuous, by Lemma \ref{lemma:convolution_convergence} we get
\[ 
\nabla K\ij^{(n)}\to\nabla K\ij, 
\]
uniformly on compact sets in $\R^d$. Since
\[
   \abss{ \sum_{j=1}^N \big[ \nabla K\ij^{(n)}\ast \rho\j^{\varepsilon_k} (t,x) - \nabla K\ij\ast \rho\j^{\varepsilon_k} (t,x)}  \leq \sup_x \bigg[  \abss{ \sum_{j=1}^N \nabla K\ij^{(n)}(x)-\nabla K\ij(x)}  \bigg], 
\]
we have that for any compact set $A\subset\R^d,$
\[
    \sum_{j=1}^N \nabla K\ij^{(n)}\ast\rho\j^{\varepsilon_k} \prt{ t,x}  \xrightarrow{n\to\infty} \sum_{j=1}^N \nabla K\ij \ast\rho\j^{\varepsilon_k} \prt{t,x},
\]
uniformly for $t\in [0,T)$, for $x\in A$, $k\in\mathbb{N}$. Therefore, considering the limit as $n\to\infty$ in \eqref{eq:kinetic_th_uniflim_rho} on a compact set $A\subset\R^d$, we get
\begin{align*}
    \abss{ \sum_{j=1}^N & \big[ \nabla K\ij\ast\rho\j^{\varepsilon_k}(t,x)  -\nabla K\ij\ast\rho\j^{\varepsilon_k}(s,y) \big] } \\& \leq (C(T)+1) \bigg( \sum_{j=1}^N  \norm{\nabla K\ij}_{W^{1,\infty}} \bigg) \big( \abs{t-s}+\abs{x-y}\big).
\end{align*}
Thus, by Ascoli-Arzelà theorem, there exist $N$ subsequences still denoted by $\rho\i^{\varepsilon_k}$, as $\fori$, such that
\[
   \sum_{j=1}^N \nabla K\ij\ast\rho\j^{\varepsilon_k} \to \sum_{j=1}^N  \nabla K\ij \ast\rho\j, 
    \]
as $\varepsilon_k\to 0$, strongly in $\ninf ([0,T)\times A)$, with $A\subset \R^d$ compact set. 
Hence, for every $t\in [0,T)$,
\begin{align*}
     \abss{ \int_{\R^d} & \nabla \Psi\i \cdot \bigg[ \sum_{j=1}^N (\nabla K\ij\ast\rho\j^{\varepsilon_k} )\rho\i^{\varepsilon_k} -\sum_{j=1}^N(\nabla K\ij\ast\rho\j)\rho\i \bigg] \,dx } 
    \\ 
    \leq & \sum_{j=1}^N \norm{ \nabla K\ij\ast\rho\j^{\varepsilon_k} -\nabla K\ij\ast\rho\j} _{\ninf (\Omega _1(T))} \norm{\nabla \Psi\i}_\ninf \\ 
    & \quad + \sum_{j=1}^N\abss{ \int_{\Omega_1(T)} \nabla \Psi\i \cdot (\nabla K\ij\ast\rho\j)(\rho\i^{\varepsilon_k}-\rho\i ) \,dx},
\end{align*}
and the first term goes to zero as $\varepsilon_k\to 0$ uniformly on $[0,T)$ and the second integral vanishes as $\varepsilon_k\to 0$ by \eqref{eq:kinetic_th_vanish_meas}. Combining this with \eqref{eq:first_integral_rho} we obtain that, for each $t\in (0,T),$
\[ \iint_{\R^{2d}} \nabla\Psi\i \cdot vf\i^{\varepsilon_k} \,dx\,dv \to -\int_{\R^d} \nabla\Psi\i \cdot \bigg[ \sum_{j=1}^N \prt{ \nabla K\ij\ast\rho\j }\rho\i \bigg] \,dx \]
as $\varepsilon_k\to 0$. Finally, define
\[\Omega_2(T)=\{ v\in\R^d \, :\, (x,v)\in \Omega(T) \}. \]
We have that $\Omega_2(T)$ is bounded for all $t\in (0,T)$ and the following uniform estimate holds:
\[ \abss{ \iint_{\R^{2d}} \nabla \Psi\i \cdot vf\i^{\varepsilon_k}\,dx\,dv } \leq D\i \norm{\nabla \Psi\i} _{\ninf (\R^d)}, \]
where the constant $D\i$ depends only on $\Omega_2(T).$ This implies, by Lebesgue's dominated convergence theorem, that
\[ \int_0^T \iint_{\R^{2d}} \nabla \Psi\i \cdot vf\i^{\varepsilon_k}\,dx\,dv\,dt\to - \int_0^T \int_{\R^d} \nabla \Psi\i \cdot \bigg[  \sum_{j=1}^N \prt{ \nabla K\ij\ast\rho\j } \rho\i \bigg]\, dx\,dt \]
as $\varepsilon_k\to 0$. Thus the limiting $N$-tuple of measures $\ro\in \mc([0,T);\mathcal{P}(\R)^N)$ is a solution to system \eqref{eq:system_compact} in the weak sense.
\end{proof}

We now give two corollaries concerning the uniqueness of solutions to system \eqref{eq:mainsystem}.
 
\begin{corollary} Assume that the assumptions in Theorem \ref{th:kinetic_limit} and Proposition \ref{prop:stability} hold. Then, the $N$-tuple $\ro \in \mc([0,T); \mathcal{P}_1 (\R^d)^N)$ obtained in Theorem \ref{th:kinetic_limit} is the unique solution to system \eqref{eq:mainsystem}.
\end{corollary}
\begin{proof}
The proof follows by Proposition \ref{prop:stability}. Indeed, if we assume that there are two solutions starting from the same initial datum, by \eqref{eq:stability} we have the statement.
\end{proof}

\begin{corollary} Assume that assumptions in Theorem \ref{th:kinetic_limit} hold. Moreover, assume that the cross-interaction kernels are equal, i.e., $H\coloneqq K\ij$, for all $i\neq j$. Then the solution to system \eqref{eq:mainsystem} obtained in Theorem \ref{th:kinetic_limit} is unique.
\end{corollary}
\begin{proof}
Since $\ro \in \mc([0,T), \mathcal{P}_1(\R^d)^N)$ is a weak solution to \eqref{eq:system_compact}, by \cite[Theorem 5.1]{fs} and the references therein, we can say that $\ro$ is the push-forward of $\ro_0$ via the flow $\T^t_{\E[\f]}$ where $\E[\f] = \prt{E\i [\f] }_{i=1}^N$ with
\[
E\i [\f] =-\sum_{j=1}^N \nabla K\ij\ast\rho\j \in \ninf ([0,T)\times \R^d), 
    \] 
that is
\[ \ro = \T^t_{\E[\f]}\#\ro_0. \]
Furthermore, $\ro(t,\cdot)$ has compact support and it is narrowly continuous in time, since we get that $\ro(t,\cdot)\in \mc([0,T);\mathcal{P}_1(\R^d)^N)$ where the continuity is in the $\W_1$ metric, (see Proposition \ref{prop:convergence_properties}). Then $\ro$ is the unique solution to \eqref{eq:system_compact} in the mass transportation sense.
\end{proof}

\section{Singular interaction potentials}
\label{sec:singular_model}

In this section, we investigate the case of singular self-interaction potentials and smooth cross-potentials and provide the details of the proof of Theorem \ref{thm_ktoc}. For this, we first discuss the existence of solutions to the coupled kinetic and first order macroscopic equations. We recall the kinetic and  macroscopic order systems:
\begin{equation}\label{eq:singular_main}
\partial_t f\i +v\cdot \nabla_x f\i =\frac{1}{\varepsilon} \nabla_v \cdot (vf\i ) + \frac{1}{\varepsilon} \bigg( \sum_{j=1}^N \nabla K\ij\ast \rho\j \bigg) \cdot \nabla_v f\i,
\end{equation}
for $i=1,\dots, N$, where $\rho\i (t,x)$ is the macroscopic population density of the $i$-th species, i.e.,
\[ \rho \i (t,x) = \int_{\R^d} f\i (t,x,v)\,dv \]
and 
\begin{equation}\label{eq:singular_limit}
\begin{dcases}
\partial_t \rho_i =  \nabla \cdot (\rho_i u_i),  \\ 
 u_i =  \sum_{j=1}^N \nabla K_{ij}\ast \rho_j,
\end{dcases}
\end{equation}
for $\fori$. Here the cross-potentials $K\ij$, $i\neq j$, are given as in \eqref{pot} and singular self-potentials $K_{ii}$ are of the form
\begin{equation}
\label{eq:singular_potential}
K\ii(x)= \frac{C\i}{\abs{x}^\ai},
\end{equation}
with $\ai \in (0,d)$ and some positive constants $C\i$. 

In the following two subsections, we establish the existence theory for the systems \eqref{eq:singular_main} and \eqref{eq:singular_limit} satisfying required regularity conditions stated in Theorem \ref{thm_ktoc}, respectively. As mentioned in Remark \ref{rmk_sing_main}, due to some technical difficulties, we construct the global/local-in-time solutions to the systems \eqref{eq:singular_main} and \eqref{eq:singular_limit} in a little more restrictive setting.

\subsection{Existence for solution to the kinetic system}
In this subsection, motivated from \cite{choi_jeong_kinetic}, we investigate the global-in-time existence of weak solutions to the kinetic system \eqref{eq:singular_main} when $\alpha_i \in (0,d-1]$. 



We start by considering a regularised version of the system \eqref{eq:singular_main}. For this purpose, we perturb the self-potentials and consider the following system
\begin{equation}
\label{eq:singular_regular}
\partial_t f\i\de+v\cdot \nabla_x f\i\de = \frac{1}{\varepsilon} \nabla_v \cdot (vf\i\de) + \frac{1}{\varepsilon} \bigg(\sum_{j=1}^N \nabla K\ij\de \ast \rho\j\de \bigg)\cdot \nabla_v f\i\de, 
\end{equation}
for $\fori$, with
\[ K\ii\de (x) \coloneqq \frac{C\i}{\abs{x}^\ai+\delta}, \]
and
\[ \rho\i\de (t,x) \coloneqq \int_{\R^d} f\i\de (t,x,v) \,dv. \]
In system \eqref{eq:singular_regular} we set $K\ij\de \coloneqq K\ij$, for $i\neq j$, in order to keep the notation to a minimum.
Notice that the global-in-time existence and uniqueness of a weak solution to the regularised system \eqref{eq:singular_regular} follows by the results developed in Section \ref{sec:smooth}, since the force fields $\nabla K\ij\de \ast\rho\j\de$ are bounded and Lipschitz continuous. Throughout this subsection, we assume $\alpha_i \in (0,d-1]$.

\subsubsection{Uniform in $\delta$ estimates}
In this part, we gather some uniform in $\delta$ estimates that we will apply for proving the existence of solutions to system \eqref{eq:singular_main}. Let us begin with $L^\infty$ bound estimates.

\begin{lemma} Let $T>0$ and $\f\de \coloneqq (f_1\de, \ldots , f_N\de)$ be the weak solution to \eqref{eq:singular_regular} on the interval $[0,T]$ in the sense of Definition \ref{def:solution}. Then we have
\[ \sup_{0\leq t\leq T} \norm{f\i\de (\cdot , \cdot, t)}_{L^p} \leq \norm{f\de _{i0} }_{L^p} e^{\frac{d}{\varepsilon} (1-\frac{1}{p})T },  \]
for $p\in [1, +\infty)$, and
\[ \sup_{0\leq t\leq T} \norm{f\i\de (\cdot , \cdot, t)}_{L^\infty} \leq \norm{f\de _{i0} }_{L^\infty} e^{\frac{d}{\varepsilon}T }. \]
\end{lemma}
\begin{proof}
By integrating by parts with respect to $x$ and $v$, we get
\begin{align*}
\frac{d}{dt} \intdd (f\i\de)^p \,dx\,dv =& -\frac{1}{\varepsilon} p (p-1) \intdd (f\i\de) ^{p-2} \nabla _x f\i\de \cdot vf\i\de \,dx\,dv\\ 
& - \frac{1}{\varepsilon} p (p-1) \intdd (f\i\de)^{p-2} \nabla_v f\i\de \cdot \bigg( \sum_{j=1}^N K\ij\de\ast\rho\j\de \bigg) f\i\de \,dx\,dv\\
 & + p (p-1) \intdd (f\i\de)^{p-2} \nabla_x f\i\de \cdot vf\i\de \,dx\,dv.
\end{align*}
Thus, 
\[ \frac{d}{dt} \intdd (f\i\de)^p\,dx\,dv = d \frac{1}{\varepsilon} (p-1) \intdd (f\i\de)^p \,dx\,dv,  \]
for $p\in [1, +\infty)$. Therefore, by Gr\"onwall's lemma we have
\[ \norm{f\i\de (\cdot , \cdot , t)}^p_{L^p} = \norm{f\de_{i0} }^p_{L^p} e^{\frac{d}{\varepsilon} (p-1)t}. \]
Then, it follows that
\[ \sup_{0\leq t\leq T} \norm{f\i\de (\cdot , \cdot, t)}_{L^p} \leq \norm{f\de _{i0} }_{L^p} e^{\frac{d}{\varepsilon} (1-\frac{1}{p})T }, \]
for $p\in [1, +\infty)$. Sending $p\to +\infty$ in the previous line, we obtain that
\[ \sup_{0\leq t\leq T} \norm{f\i\de (\cdot , \cdot, t)}_{L^\infty} \leq \norm{f\de _{i0} }_{L^\infty} e^{\frac{d}{\varepsilon}T }, \]
that concludes the proof.
\end{proof}

Now we state a Lemma that points out the relationship between the local density and the kinetic energy (cf.\ \cite[Lemma 3.1]{golse_rey}), that we will use to estimate the interaction energy. Notice that in the next result we consider generic functions and we do not work along the solutions to system \eqref{eq:singular_regular}.

The proofs of the following two lemmas are similar to the ones in \cite[Lemma 2.2]{choi_jeong_kinetic} and \cite[Lemma 2.3]{choi_jeong_kinetic}, thus we omit the details.

 
\begin{lemma}\label{lem_rho_i} Assume that $f\i \in L^1_+ \cap L^\infty (\Rdd)$ and $\abs{v}^2f\i \in L^1 (\Rdd)$, for $\fori$. Then, there exists a positive constant $C$ such that
\[ \norm{\rho\i} _{L^\frac{d+2}{d}} \leq C \norm{f\i}_{L^\infty} ^{\frac{2}{d+2}} \bigg( \intdd \abs{v}^2 f\i \,dx\,dv \bigg)^{\frac{d}{d+2}}. \]
In particular, we find that
\[ \norm{\rho\i}_{L^p} \leq C \norm{f\i}_{L^\infty}^{\frac{2}{d+2}\beta} \bigg( \intdd \abs{v}^2 f\i \,dx\,dv \bigg) ^{\frac{d}{d+2}\beta} \norm{\rho\i}_{L^1}^{1-\beta}, \]
for all $p\in [1, \frac{d+2}{d}]$, with $\rho\i = \int_{\R^d} f\i\,dv$ and $\beta = \frac{d+2}{2} (1-\frac{1}{p})$.
\end{lemma}

Let us now provide a bound estimate on the interaction energy.

\begin{lemma}\label{lem_poten_ii} Let $T>0$ and $\f\de$ be the weak solution to \eqref{eq:singular_regular} on the interval $[0,T]$. Then
\[ \abss{\intdd K\ii\de (x-y) \rho\i\de (x) \rho\i\de(y) \,dx\,dy} \leq C\i \norm{\rho_{i0}}_{L^1} ^{2-\frac{5}{2d}\ai} \norm{\rho\i\de}_{L^{\frac{d+2}{d}}}^{\frac{5}{2d}\ai}, \]
where $C\i >0$ is independent of $\delta$.
\end{lemma}

Next we prove a uniform in $\delta$ estimate on the second moments of the weak solution $\f\de$ to system \eqref{eq:singular_regular}.

\begin{prop}
\label{prop_sing_sec_mom}
Let $T>0$ and $\f\de$ be the weak solution to system \eqref{eq:singular_regular} on the interval $[0,T]$. Assume that
\[ \int_{\R^d} \rho_{i0} K\ii \ast \rho_{i0}\,dx < \infty. \] Then the following estimate on the second moment holds:
\[ \intdd \bigg( \frac{\abs{x}^2}{2} + \frac{\abs{v}^2}{2} \bigg) f\i\de \,dx\,dv + \frac{1}{\varepsilon} \int_0^t \intdd |v|^2f\i\de \,dx\,dv\,ds \leq C, \]
for all $t\in [0,T]$ and for some $C>0$ independent of $\delta$.
\end{prop}

\begin{proof}
A direct computation gives that
\begin{align*}
&\frac{1}{2} \frac{d}{dt} \bigg( \intdd \abs{v}^2 f\i\de \,dx\,dv \bigg) \\ 
& \quad = -\frac{1}{\varepsilon} \sum_{\substack{j=1 \\ j\neq i}}^N  \intdd (\nabla K\ij\de\ast\rho\j\de ) \cdot vf\i\de \,dx\,dv   - \frac{1}{\varepsilon} \intdd (\nabla K\ii\de \ast\rho\i\de) \cdot vf\i\de \,dx\,dv \\ 
&\qquad - \frac{1}{\varepsilon} \intdd |v|^2f\i\de \,dx\,dv.
\end{align*}
For $i\neq j$, we have that $\abs{\nabla K\ij \ast \rho\j} \leq \norm{\nabla K\ij}_{L^\infty}$, thus
\[ \abss{ \intdd (\nabla K\ij\de\ast\rho\j\de ) \cdot vf\i\de\,dx\,dv} \leq \norm{\nabla K\ij\de}_{L^\infty} \intdd \abs{v} f\i\de\,dx\,dv. \]
If, instead, $i=j$, by using \eqref{eq:cons_mass} in Proposition \ref{prop:estimate_smooth} we get
\begin{align*}
\frac{d}{dt} \bigg( \frac{1}{2}  \intdd K\ii\de (x-y)\rho\i\de (x) \rho\i\de(y) \,dx\,dy \bigg) 
=& \intdd K\ii\de (x-y) \partial_t \rho\i\de (x) \rho\i\de (y) \,dx\,dy \\
=& \intdd (\nabla K\ii\de \ast\rho\i\de) \cdot vf\i\de \,dx\,dv.
\end{align*}
Thus, we have
\begin{align*}
\frac{1}{2} \frac{d}{dt} \bigg( & \intdd \abs{v}^2 f\i\de \,dx\,dv \bigg) + \frac{1}{2\varepsilon} \frac{d}{dt} \bigg( \intdd K\ii\de (x-y) \rho\i\de (x) \rho\i\de (y) \,dx\,dy \bigg) \\
=& -\frac{1}{\varepsilon} \sum_{\substack{j=1 \\ j\neq i}}^N \intdd ( \nabla K\ij\ast\rho\j\de ) \cdot vf\i\de\,dx\,dv - \frac{1}{\varepsilon} \intdd |v|^2f\i\de\,dx\,dv.
\end{align*}
In the spatial variable, we have the following estimate for the second order moment
\begin{align*}
\frac{d}{dt} \bigg( & \intdd \frac{\abs{x}^2}{2} f\i\de \,dx\,dv \bigg) = \intdd \frac{\abs{x}^2}{2} \partial_t f\i\de \,dx\,dv \\
=& \intdd x\cdot vf\i\de\,dx\,dv \leq \intdd \bigg( \frac{\abs{x}^2}{2} + \frac{\abs{v}^2}{2} \bigg) f\i\de \,dx\,dv.
\end{align*}
Now, considering the estimates above, we obtain
\begin{align*}
&\intdd \bigg( \frac{\abs{x}^2}{2}+\frac{\abs{v}^2}{2} \bigg) f\i\de\,dx\,dv +  \frac{1}{2\e} \intdd K\ii\de (x-y)\rho\i\de (x)\rho\i\de (y)\,dx\,dy \\
& \qquad + \frac{1}{\varepsilon} \int_0^t \intdd \frac{1}{f\i\de} \abs{ vf\i\de}^2 \,dx\,dv\,ds \\
& \quad \leq  \intdd \bigg( \frac{\abs{x}^2}{2}+\frac{\abs{v}^2}{2} \bigg) f_{i0}\de \,dx\,dv + \frac{1}{2\e} \intdd K\ii\de (x-y)\rho_{i0}\de (x)\rho_{i0}\de (y)\,dx\,dy \\
& \qquad + \frac{1}{\varepsilon} \sum_{\substack{j=1 \\ j\neq i}}^N \int_0^t \intdd ( \nabla K\ij\de\ast\rho\j\de ) \cdot vf\i\de\,dx\,dv  + \int_0^t \intdd \bigg( \frac{\abs{x}^2}{2} + \frac{\abs{v}^2}{2} \bigg) f\i\de\,dx\,dv\,ds.
\end{align*}
By Lemmas \ref{lem_rho_i} and \ref{lem_poten_ii}, we know that
\[ \abss{ \intdd K\ii\de (x-y)\rho\i\de (x)\rho\i\de (y)\,dx\,dy } \leq C, \]
where $C$ is independent of $\delta$.
We get
\begin{align*}
\intdd & \bigg( \frac{\abs{v}^2}{2}+\frac{\abs{x}^2}{2} \bigg) f\i\de\,dx\,dv + \frac{1}{\varepsilon} \int_0^t \intdd \frac{1}{f\i\de} \abs{ vf\i\de }^2\,dx\,dv\,ds \\
\leq &\intdd \bigg( \frac{\abs{x}^2}{2}+\frac{\abs{v}^2}{2} \bigg) f_{i0}\de \,dx\,dv+ \ C \int_0^t \intdd ( \abs{v}^2 + \abs{x}^2 ) f\i\de\,dx\,dv\,ds +C
\end{align*}
for some $C>0$ independent of $\delta$. Then, by Gr\"onwall's lemma we obtain the result.
\end{proof}

\begin{rmk} From Proposition \ref{prop_sing_sec_mom}, we deduce the following estimates on total energy for $f\i\de$, as $\fori$, and for all $t \in [0,T]$:
\begin{align*}
\begin{aligned}
&\sum_{i=1}^N\intdd \frac{\abs{v}^2}2 f\de_i \,dx\,dv + \frac1{2\e} \sum_{i=1}^N \intr \rho\de_i K_{ii} *\rho\de_i\,dx + \frac{1}{\e}\int_0^t \sum_{i=1}^N\intdd \abs{v}^2 f\de_i\,dx\,dv\,ds \\
&\quad \leq \sum_{i=1}^N\intdd \frac{\abs{v}^2}2 f\de_{i0} \,dx\,dv + \frac1{2\e} \sum_{i=1}^N \intr \rho\de_{i0} K_{ii} *\rho\de_{i0}\,dx +\frac1{\e} \sum_{i \neq j}\norm{\nabla K_{ij}}_{L^\infty} t.
\end{aligned}
\end{align*}
\end{rmk}


\subsubsection{Existence of weak solution to the kinetic system}
\label{sec:singular_existence}
Now, we prove the existence of weak solutions to system \eqref{eq:singular_main}. For this purpose, we need the following lemma, cf.\ \cite{glassey,karper_mellet_trivisa}.

\begin{lemma}
\label{lemma:rel_comp}
Let $\{ f^n\}_n$ be bounded in $L^p_{loc} ([0,T]\times\Rdd)$ with $1<p<\infty$, and $ \{ G^n\}_n$ be bounded in $L^p_{loc} ([0,T]\times\Rdd)$. Assume that $f^n$ and $G^n$ satisfy 
\[ \partial_t f^n + v \cdot \nabla_x f^n = \nabla_v G^n, \qquad f^n \mid_{t=0} = f_0 \in L^p (\Rdd), \]
and
\[ f^n \mbox{ is bounded in } L^\infty (\Rdd), \]
\[ (\abs{v}^2 + \abs{x}^2 ) f^n \mbox{ is bounded in } L^\infty ((0,T);L^1 (\Rdd)). \]
Then, for any $q < \frac{d+2}{d+1}$, the sequence
\[ \bigg\{ \int_{\R^d} f^n \,dv \bigg\}_n \]
is relatively compact in $L^q ((0,T)\times \R^d)$.
\end{lemma}

The existence result of weak solutions to system \eqref{eq:singular_main} is contained in the following theorem.

\begin{theorem} 
\label{th:existence_kin_singular}
Assume that the initial datum $\f_0$ satisfies 
\[ f_{i0} \in L^1_+ \cap L^\infty (\Rdd), \qquad (\abs{x}^2 + \abs{v}^2) f_{i0} \in L^1 (\Rdd), \]
and
\[ (K\ii\ast \rho_{i0})f_{i0} \in L^1 (\Rdd). \]
Then there exists a weak solution $\f$ to \eqref{eq:singular_main} such that
\[ \f\in \mc([0,T]; \mathcal{P}(\Rdd)^N). \]
\end{theorem}
\begin{proof}
By the uniform in $\delta$ bound estimates obtained above we know
\[ \norm{f\i\de}_{L^\infty ((0,T); L^p(\Rdd))} + \norm{\rho\i\de}_{L^\infty ((0,T); L^q(\R^d))} \leq C, \]
with $p\in [1+\infty ]$, $q\in [1, \frac{d+2}{d}]$, $C>0$ independent of $\delta$. Therefore, by compactness theory, we have that as $\delta \to 0$, up to a subsequence,
\begin{gather*}
f\i\de \xrightharpoonup{\ast} f\i \quad \mbox{in} \quad L^\infty ((0,T); L^p(\Rdd)), \quad p\in [1,+\infty], \\
\rho\i\de \xrightharpoonup{\ast} \rho\i \quad \mbox{in} \quad L^\infty ((0,T); L^p(\R^d)), \quad p\in [1,\tfrac{d+2}{d}].
\end{gather*}
Set 
\[ G\i\de \coloneqq \frac{1}{\varepsilon} vf\i\de + \frac{1}{\varepsilon} \sum_{j=1}^N (\nabla K\ij\de\ast \rho\j\de )f\i\de. \]
We want to prove that $G\i\de \in L^p_{loc} ([0,T]\times\Rdd)$ for some $p\in (1,\infty)$, in order to apply Lemma \ref{lemma:rel_comp}.
We need to check the self-interaction terms.
Let $q<2$. Then
\begin{align*}
\intdd \abs{vf\i\de}^q \,dx\,dv = & \intdd \bigg(|v|^2f\i\de  \bigg)^\frac{q}{2} (f\i\de)^\frac{q}{2} \,dx\,dv\leq  \bigg( \intdd |v|^2f\i\de  \,dx\,dv \bigg)^\frac{q}{2} \norm{f\i\de}^\frac{q}{2}_{L^\frac{q}{2-q}}.
\end{align*}
For the second term, by using Young's inequality for convolution, we obtain that
\[ \norm{(\nabla K\ii\de \ast\rho\i\de) f\i\de }_{L^p} \leq C \norm{f\i\de}_{L^\infty} \norm{ \nabla K\ii \ast \rho\i\de}_{L^p} \leq C \norm{f\i\de}_{L^\infty} \norm{\rho\i\de}_{L^p}, \]
for $p < \frac{d+2}{d}$. Thus, by Lemma \ref{lemma:rel_comp}, we get
\[ \rho\i\de \to \rho\i \mbox{ in } L^q ((0,T)\times\R^d) \mbox{ and a.e.}, \]
up to a subsequence, as $\delta \to 0$, for $q < \frac{d+2}{d+1}$.
Now we want to prove that
\[ (\nabla K\ii\de \ast\rho\i\de) f\i\de \to (\nabla K\ii\ast\rho\i) f\i, \]
in the sense of distributions.
Let $\Psi\i \in \mc_c^\infty ([0,T]\times \Rdd)$.
\begin{align*}
\int_0^T &\intdd [  (\nabla K\ii\de \ast\rho\i\de )f\i\de -(\nabla K\ii\ast\rho\i)f\i] \Psi\i\,dx\,dv\,ds \\
=& \int_0^T \int_{\R^d} (\nabla (K\ii\de - K\ii)\ast\rho\i) \rho_{i,\Psi} \,dx\,ds  + \int_0^T \int_{\R^d} \nabla K\ii\de \ast (\rho\i\de-\rho\i) \rho\de_{i,\Psi} \,dx\,ds \\
&\quad+ \int_0^T \int_{\R^d} (\nabla K\ii\de\ast\rho\i) (\rho\de_{i,\Psi}-\rho_{i,\Psi} ) \,dx\,ds \\
\eqqcolon & I+II+III,
\end{align*}
with $\rho_{i,\Psi} \coloneqq \int_{\R^d} f\i \Psi\, dv$ and $\rho\de_{i,\Psi} \coloneqq \int_{\R^d} f\i\de \Psi \,dv$.
Thanks to the uniform in $\delta$ estimate for $f\i\de$ in $L^\infty ((0,T)\times\Rdd)$ and the compact support of $\Psi\i$, we find
\[ \rho_{i,\Psi}, \ \rho\de_{i,\Psi} \in L^p((0,T); L^q (\R^d)), \]
for any $p,q \in [1,\infty]$, uniformly in $\delta$.

\noindent\textbf{Estimate of $I$.} We have that $\abs{(\nabla K\ii\de \ast\rho\i)\rho_{i,\Psi}} \leq \abs{\nabla K\ii \ast \rho\i} \abs{\rho_{i,\Psi}}$ and $(\nabla K\ii\de \ast\rho\i) \rho_{i,\Psi}$ converges pointwise to $(\nabla K\ii\ast\rho\i)\rho_{i,\Psi}$ as $\delta \to 0$. Moreover, by Hardy-Littlewood-Sobolev inequality, we get
\begin{align*}
\int_0^T \int_{\R^d} (\abs{\nabla K\ii} \ast \rho\i) \abs{\rho_{i,\Psi}} \,dx\,ds & \leq \int_0^T \intdd \rho\i(x) \abs{x-y}^{-(\ai+1)} \abs{\rho_{i,\Psi}}(y) \,dx\,dy\,ds \\
& \leq C \norm{\rho\i}_{L^{p_i}(\R^d \times (0,T))} \norm{\rho_{i,\Psi}}_{L^{{p_i}'} ((0,T); L^{q_i} (\R^d))},
\end{align*}
where 
\begin{equation}\label{index_pi}
p_i\in \left(1,\frac{d+2}{d}\right), \quad \frac{\ai+1}{d}=1-\frac{1}{q_i}+\frac{1}{p_i}, 
\end{equation}
and $p_i'$ is the H\"older conjugate of $p_i$. Therefore, by Lebesgue's dominated convergence theorem, we obtain that $I$ vanishes as $\delta\to 0$. 

\noindent\textbf{Estimate of $II$.} As in the previous estimate, we have that
\begin{align*}
\abss{ \int_0^T \int_{\R^d} \nabla K\ii \ast (\rho\i\de -\rho\i) \rho_{i,\Psi}\de  \,dx\,ds}  & \leq \int_0^T \intdd \abs{\rho\i\de - \rho\i}(x) \abs{x-y}^{-(\ai+1)} \abs{\rho\de_{i,\Psi}} (y)\,dx\,dy\,ds \\
& \leq C \norm{\rho\i\de - \rho\i} _{L^{p_i} (\R^d \times (0,T))} \norm{\rho_{i\Psi}}_{L^{{p_i}'} ((0,T); L^{q_i} (\R^d))},
\end{align*}
with $p_i\in (1,\frac{d+2}{d+1})$ and $q_i$ is chosen as in \eqref{index_pi}. Thus, $II\to 0$ as $\delta \to 0$.

\noindent\textbf{Estimate of $III$.} As said, we know that
\[ (\nabla K\de\ii \ast \rho\i ) \Psi \in L^1 ((0,T); L^q (\R^d)), \]
with $q<2$ uniformly in $\delta$. Then, since $f\i\de \xrightharpoonup{\ast} f\i$, we obtain that $III\to 0$ as $\delta \to 0$. 
We conclude that $\f$ is a weak solution to system \eqref{eq:singular_main}.
\end{proof}

\subsection{Existence and uniqueness of regular solutions to the macroscopic system}
Our aim here is to prove that, under suitable assumptions on the parameter, there exists a unique regular solution to system \eqref{eq:singular_limit}. We only focus on the case $-2 \leq \ai - d\leq 0$ for all $\fori$. Note that the case $\alpha_i \leq d - 2$ for all $i=1,\dots, N$ can be handled by the classical well-posedness theory. Moreover, the case $\alpha_i \leq d - 2$ for some $i$ can be taken into account by a simple modification of our arguments.

Here, we assume that the cross-interaction kernels $K_{ij}$, $i \neq j$, satisfy
\[
\nabla K_{ij} \in W^{1,1}(\R^d).
\]

Let us first recall Moser-type inequalities.
\begin{lemma}\label{lem_moser}
\begin{itemize}
    \item[(i)] Let $s>0$, $r \in (1,\infty)$, and $p_1, p_2, q_1, q_2 \in (1,\infty]$. Then we have
    \[
    \|\Lambda^s(fg)\|_{L^r} \lesssim \|\Lambda^s f\|_{L^{p_1}}\|g\|_{L^{q_1}} + \|f\|_{L^{p_2}}\|\Lambda^s g\|_{L^{q_2}}
    \]
    for $f \in \dot{W}^{s,p_1} \cap L^{p_2}$ and $g \in \dot{W}^{s,q_2} \cap L^{q_1}$, where $\frac1r = \frac1{p_1} + \frac1{q_1} = \frac1{p_2} + \frac1{q_2}$.
    \item[(ii)] Let $s >0$. If $f \in \dot{W}^{1,\infty} \cap \dot{H}^s$ and $g \in L^\infty \cap \dot{H}^{s-1}$, then we obtain
    \[
    \| [\Lambda^s, f]g\|_{L^2} \lesssim \|f\|_{\dot{H}^s}\|g\|_{L^\infty} + \|\nabla f\|_{L^\infty}\|g\|_{\dot{H}^{s-1}}.
    \]
\end{itemize}
Here $[\cdot, \cdot]$ denotes the commutator operator, i.e., $[A,B] = AB - BA$.
\end{lemma}

We then present a priori estimate of solutions of \eqref{eq:singular_limit} in the lemma below.
\begin{lemma}\label{lem_apriori} Let $d\geq 1$, $s> \frac{d}{2}+3$, and $-2 \leq \ai - d\leq 0$ for $\fori$. Let $\ro$ be a sufficiently smooth solution to \eqref{eq:singular_limit} decaying fast at infinity on a time interval $[0,T]$. Then by choosing a sufficiently small $T^* > 0$ depending on $\ro_0$, we have
    \[
    \|\ro(t)\|_{H^s} \leq 4\|\ro_0\|_{H^s}.
    \]
\end{lemma}
\begin{proof} We first notice that our main system can be rewritten as
    \[
    \begin{dcases}
    \partial_t \rho\i + \nabla \cdot (\rho\i u\i)=0,\\
    u\i= - \Lambda ^{\alpha_i - d} \nabla \rho\i - \nabla V\i[\ro],
    \end{dcases}
    \]
for $\fori$, with 
\[
V\i[\ro] = \sum_{j\neq i} K\ij \ast \rho\j\, .
\]
Then for $s> \frac{d}{2}+3$, we find
\begin{align*}
 \frac12\frac{d}{dt}\|\Lambda^s \rho_i\|_{L^2}^2 &= \intr \Lambda^s \rho_i \cdot \Lambda^s (\nabla \cdot (\rho_i \Lambda ^{\alpha_i - d} \nabla \rho_i))\,dx + \intr \Lambda^s \rho_i \cdot \Lambda^s (\nabla \cdot (\rho_i \nabla V_i[\ro])\,dx\cr
 &=: I + II,
\end{align*}
where $I$ can be estimated as
\[
|I| \leq C\|\rho_i\|_{H^s}^3 \leq C\|\ro\|_{H^s}^3.
\]
due to a direct consequence of \cite{choi_jeong_fractional}. 
Next, note that
\begin{equation}\label{Linf_V}
\|\nabla^2 V_i[\ro]\|_{L^\infty} \leq \sum_{j\neq i}\|\nabla^2 K_{ij} \ast \rho_j\|_{L^\infty} \leq \sum_{j\neq i} \|\nabla^2 K_{ij}\|_{L^1} \|\rho_j\|_{L^\infty} \leq C\sum_{j\neq i}\|\rho_j\|_{H^s} \leq C\|\ro\|_{H^s}
\end{equation}
and
\[
\|\nabla^2 V_i[\ro]\|_{H^s} \leq \sum_{j\neq i}\|\nabla^2 K_{ij} \ast \rho_j\|_{H^s} \leq \sum_{j\neq i}\|\nabla^2 K_{ij}\|_{L^1}\|\rho_j\|_{H^s}\leq C\|\ro\|_{H^s}.
\]
Then we estimate
\begin{align*}
II &= \intr (\Lambda^s \rho_i) \Lambda^s (\nabla \rho_i \cdot \nabla V_i[\ro] + \rho_i \Delta V_i[\ro])\,dx\cr
&= \intr (\Lambda^s \rho_i) \nabla (\Lambda^s \rho_i) \cdot \nabla V_i[\ro]\,dx + \intr (\Lambda^s \rho_i) [\Lambda^s, \nabla V_i[\ro]]\nabla \rho_i\,dx + \intr (\Lambda^s \rho_i)\Lambda^s (\rho_i \Delta V_i[\ro])\,dx\cr
&\leq \|\Delta V_i[\ro]\|_{L^\infty} \|\Lambda^s \rho_i\|_{L^2}^2 + \|\Lambda^s \rho_i\|_{L^2}(\|[\Lambda^s, \nabla V_i[\ro]]\nabla \rho_i\|_{L^2} + \|\Lambda^s (\rho_i \Delta V_i[\ro])\|_{L^2}).
\end{align*}
We now use Lemma \ref{lem_moser} to deduce
\[
\|[\Lambda^s, \nabla V_i[\ro]]\nabla \rho_i\|_{L^2} \lesssim \|\nabla V_i[\ro]\|_{H^s}\|\nabla \rho_i\|_{L^\infty} + \|\nabla^2 V_i[\ro]\|_{L^\infty}\|\nabla \rho_i\|_{H^{s-1}} \leq C\|\ro\|_{H^s}^2 
\]
and
\[
\|\Lambda^s (\rho_i \Delta V_i[\ro])\|_{L^2} \lesssim \|\Lambda^s \rho_i\|_{L^2}\|\Delta V_i[\ro]\|_{L^\infty} + \|\rho_i\|_{L^2}\|\Lambda^s \Delta V_i[\ro]\|_{L^\infty} \leq C\|\ro\|_{H^s}^2.
\]
Hence, by combining all of the above estimates, we have
\[
\frac{d}{dt}\|\ro\|_{H^s} \leq C\|\ro\|_{H^s}^2,
\]
and subsequently, by choosing a $T^* > 0$ small enough, we conclude the desired result.
\end{proof}

We now present the result on the existence and uniqueness of regular solutions to system \eqref{eq:singular_limit}.
\begin{theorem}\label{thm_classical}Let $d\geq 1$, $s> \frac{d}{2}+3$, and $-2 \leq \ai - d\leq 0$ for $\fori$. Then, the system \eqref{eq:singular_limit} admits a local-in-time unique solution $\rho_i \in L^1 \cap H^s (\R^d)$, i.e., for any non-negative initial datum $\ro_0 \in (L^1\cap H^s (\R^d))^N$, there exists $T=T(\ro_0)>0$ and some non-negative solutions $\ro \in \mc ([0,T); L^1\cap H^s (\R^d))^N$ to \eqref{eq:singular_limit} with $\ro(t=0)=\ro_0$.
\end{theorem}
\begin{proof}
{\it (existence)}: The existence of solutions $\rho_i \in L^\infty(0,T^*; L^1 \cap H^s (\R^d))$ can be obtained by combining the a priori estimate in Lemma \ref{lem_apriori} and the classical approximation arguments. Thus, we omit its details here, and refer to \cite{choi_jeong_fractional} for readers who are interested in it. 

{\it (uniqueness)}: Let $\rho_i$ and $\tilde \rho_i$ be two solutions to \eqref{eq:singular_limit} obtained in the above. We also let $v_i$ and $\tilde v_i$ be corresponding vector fields. Then defining $g_i := \rho_i - \tilde \rho_i$ and $u_i = v_i - \tilde v_i$, we find that $g_i$ satisfies
\[
\partial_t g_i + \nabla \cdot (g_i v_i + \tilde \rho_i u_i) = 0.
\]
We now estimate
\[
\frac{d}{dt}\|g_i\|_{L^2}^2 \leq C(\|\nabla v_i\|_{L^\infty}\|g_i\|_{L^2}^2 + \|\nabla \tilde \rho_i\|_{L^\infty}\|g_i\|_{L^2}\|g_i\|_{H^1} + \|\tilde \rho_i\|_{L^\infty}\|g_i\|_{H^1}^2),
\]
thanks to \eqref{Linf_V} and $s > \frac d2 + 3$.
Since 
\[
\|\nabla v_i\|_{L^\infty} \lesssim \|\Lambda^{\alpha_i - d}\nabla^2 \rho_i\|_{L^\infty} + \|\nabla^2 V_i[\ro]\|_{L^\infty} \lesssim \|\ro\|_{H^s},
\]
we arrive at
\begin{equation}\label{g_i_est}
\frac{d}{dt}\|g_i\|_{L^2}^2 \leq C\|g_i\|_{H^1}^2.    
\end{equation}
We next estimate $\|\nabla g_i\|_{H^1}$. Note that
\begin{align*}
\frac12\frac{d}{dt}\|\partial g_i\|_{L^2}^2 &= - \intr \nabla \cdot (\partial g_i v_i) \partial g_i\,dx - \intr \nabla \cdot (g_i \partial v_i) \partial g_i\,dx \cr
&\quad - \intr \nabla \cdot (\partial \tilde \rho_i u_i) \partial g_i\,dx - \intr \nabla \cdot (\tilde \rho_i \partial u_i) \partial g_i\,dx.
\end{align*}
Following \cite{choi_jeong_fractional}, we deduce that the right hand side of the above can be bounded by
\[
C(\|\ro\|_{H^s} + \|\tilde \ro\|_{H^s})\|g_i\|_{H^1}^2 + (\|\nabla \cdot (\partial \tilde \rho_i \nabla V_i[\ro - \tilde\ro])\|_{L^2} + \|\nabla \cdot (\tilde \rho_i \partial \nabla V_i[\ro - \tilde\ro])\|_{L^2})\|g_i\|_{H^1}.
\]
Here we further estimate
\begin{align*}
&\|\nabla \cdot (\partial \tilde \rho_i \nabla V_i[\ro - \tilde\ro])\|_{L^2} + \|\nabla \cdot (\tilde \rho_i \partial \nabla V_i[\ro - \tilde\ro])\|_{L^2} \cr
&\quad \leq C\|\nabla^2 \tilde \rho_i\|_{L^\infty}\|\nabla V_i[\ro - \tilde\ro]\|_{L^2} + \|\partial \tilde \rho_i\|_{L^\infty}\|\nabla^2 V_i[\ro - \tilde\ro]\|_{L^2} + \|\tilde \rho_i\|_{L^\infty}\|\nabla^3 V_i[\ro - \tilde\ro]\|_{L^2}\cr
&\quad \leq C\sum_{j \neq i}\|g_j\|_{H^1}\|\tilde \ro\|_{H^s}
\end{align*}
to yield
\[
\frac12\frac{d}{dt}\sum_{i=1}^N\|\partial g_i\|_{L^2}^2 \leq C\sum_{i=1}^N\|g_i\|_{H^1}^2.
\]
We finally combine the above and \eqref{g_i_est} to conclude that
\[
\rho_i \equiv \tilde \rho_i
\]
provided that $\rho_i(0) = \tilde\rho_i(0)$ for all $i=1,\dots, N$. This completes the proof.
\end{proof}
\begin{rmk} Here we sketch another way of showing the existence and uniqueness of regular solutions to \eqref{eq:singular_limit}. 
We first approximate the system \eqref{eq:singular_limit} by
    \[
    \begin{dcases}
    \partial_t \rho\i^{n+1} + \nabla \cdot (\rho\i^{n+1} u\i^n)=0,\\
    u\i^n= - \Lambda ^{\alpha_i - d} \nabla \rho\i^n - \nabla V\i[\ro^n],
\end{dcases}
\]
for $\fori$, with 
\[
V\i[\ro^n] = \sum_{j\neq i} K\ij \ast \rho\j^n\, .
\]
Here  the initial datum and first iteration step are given by
\[
\rho_i^n(0,x) = \rho_i(0,x) \quad \mbox{ for all $i=1,\dots,N $ and for all $n \geq 1$}
\]
and
\[
\rho_i^0(t,x) = \rho_i(0,x) \quad \mbox{ for all $i=1,\dots,N$.}
\]
We then combine the a priori estimate in Lemma \ref{lem_apriori} and the uniqueness estimate in the proof of Theorem \ref{thm_classical} to show that the approximate sequence $\{\rho_i^n\}_{n \in \N}$ is Cauchy in $L^\infty(0,T^*; H^1)$. From which, we have the existence of limit function $\rho_i$ and show that function is the solution to \eqref{eq:singular_limit}. 
\end{rmk}

%
%
%
%
%
%
%
\subsection{Small inertia limit}
\label{sec:singular_small_inertia}
In this subsection, we show the rigorous small inertia limit of the system \eqref{eq:singular_main} providing Theorem \ref{thm_ktoc}. Since we want to study the behaviour of solutions to kinetic system \eqref{eq:singular_main} with respect to the inertia parameter $\varepsilon >0$, we explicit the $\varepsilon$-dependence, namely we define $\f^\e = (f\i^\e)_{i=1}^N$ to be a weak solution to the system:
\begin{equation}
\label{eq:singular_eps}
\partial_t f\i^\e +v\cdot \nabla_x f\i^\e =\frac{1}{\varepsilon} \nabla_v \cdot (vf\i^\e ) + \frac{1}{\varepsilon} \bigg( \sum_{j=1}^N \nabla K\ij\ast \rho\j^\e \bigg) \cdot \nabla_v f\i^\e,
\end{equation}
for $\fori$, with smooth cross-potentials as in \eqref{pot} and singular self-potentials of the form \eqref{eq:singular_potential}. 

We quantitatively show the convergence of solutions $\f^\e$ towards the system  
\begin{equation}
\label{eq:singular_lim}
\begin{dcases}
\partial _t \rho\i = \nabla \cdot (\rho\i u\i),\\
u\i = \sum_{j=1}^N \nabla K\ij \ast \rho\j,
\end{dcases}
\end{equation}
for $\fori$, as $\e \to 0$.

In order to measure the error between solutions to the systems \eqref{eq:singular_eps} and \eqref{eq:singular_lim}, we employ the modulated energy method. For this, we first recall from \cite{CJpre} (see also \cite{NRSpre}) the following modulated interaction energy estimates. 

\begin{theorem}\label{thm_mod}
Let $T>0$ and $K$ be given by 
\[
K(x) = \frac1{\abs{x}^\alpha} \quad \mbox{with } \alpha \in (0,d).
\] 
Suppose that the pairs $(\bar\rho, \bar u)$ and $(\rho,u)$ satisfy the followings:
\begin{enumerate}
\item[(i)]
$(\bar\rho, \bar u)$ and $(\rho,u)$ satisfy the continuity equations in the sense of distribution:
\[
\partial_t \bar\rho + \nabla \cdot (\bar\rho \bar u) =0 \quad \mbox{and} \quad \partial_t \rho + \nabla \cdot (\rho u) =0,
\]
\item[(ii)]
$(\bar\rho, \bar u)$ and $(\rho,u)$ satisfy the energy inequality:
\[
\sup_{0\le t \le T}\bigg( \intr \bar\rho \abs{\bar u}^2\,dx + \intr \bar\rho K\ast\bar\rho\,dx\bigg) <\infty,
\] and
\[ \sup_{0\le t \le T}\bigg( \intr \rho \abs{u}^2\,dx + \intr \rho K\ast\rho\,dx\bigg) <\infty,
\]
\item[(iii)]
$\bar\rho, \rho\in \mc((0,T); L^1(\R^d))$, $\nabla u\in L^\infty(\R^d \times (0,T))$ and if $\alpha<d-2$, 
\[
\left\{\begin{array}{lcl}
\nabla^{[(d-\alpha)/2]+1} u\in L^\infty(0,T;L^{\frac{d}{[(d-\alpha)/2]}}(\R^d)) & \mbox{if} & \alpha\in(0,d-2)\setminus(d-2\N),\\
\nabla^{\frac{d-\alpha}{2}} u \in L^\infty(0,T;L^{\frac{2d}{d-\alpha-2}}(\R^d)) & \mbox{if} & \alpha\equiv d \ \emph{ mod } 2, 
\end{array}\right.
\]
where $d-2\N \coloneqq \{d-2n \ : \ n \in \N\}$ and $[ \,\cdot \, ]$ denotes the floor function. 
\end{enumerate}
Then we have
\[
\frac12\frac{d}{dt}\intr (\rho -\bar\rho) K \ast (\rho - \bar\rho)\,dx \leq \intr \bar\rho(u - \bar u) \cdot \nabla K \ast (\rho - \bar\rho)\,dx + C\intr (\rho - \bar\rho) K \ast (\rho - \bar\rho)\,dx
\]
for $t \in [0,T)$ and some $C>0$ which depends only on $\alpha$, $d$ and $\norm{\nabla u}_{L^\infty(\R^d \times (0,T))}$, and  if $d < \alpha-2$, additionally  
\[
\left\{\begin{array}{lcl}
\norm{\nabla^{[(d-\alpha)/2]+1} u}_{L^\infty((0,T);L^{\frac{d}{[(d-\alpha)/2]-1}})}, & \mbox{if} & \alpha\in(0,d-2)\setminus (d-2\N),\\
\norm{\nabla^{(d-\alpha)/2}  u}_{L^\infty((0,T);L^{\frac{2d}{d-\alpha-2}})}, & \mbox{if} & \alpha\equiv d \ \emph{mod } 2.
\end{array}\right.\]
\end{theorem}

\begin{rmk} \label{rmk:macr_vel} If we define the macroscopic velocity as
\[ 
u (t,x)= \frac{\intr vf(t,x,v)\,dv}{\intr f(t,x,v)\,dv},
\]
then 
\[
 \rho^\e_i \abs{u^\e_i}^2     \leq \intr \abs{v}^2 f^\e_i\,dv, 
\]
and thus by Proposition \ref{prop_sing_sec_mom}, we obtain that for $\e > 0$
\[
\sum_{i=1}^N\intr \rho^\e_i \abs{u^\e_i}^2 \,dx < \infty
\]
on some time interval $[0,T]$.
\end{rmk}

We also recall from \cite[Lemma 4.1]{CCJ21} (see also \cite[Proposition 3.1]{Choi21}, \cite[Theorem 23.9]{Vi09}, \cite{ags, CC21, FK19}) the following lemma which gives a relation between the $1$-Wasserstein distance and modulated kinetic energy.
\begin{lemma}\label{lem_dbl} Let $T>0$ and $\bar\rho:[0,T] \to \calP(\R^d)$ be a narrowly continuous solution of 
\[
\partial_t \bar\rho + \nabla \cdot (\bar\rho \bar u) = 0,
\]
that is, $\bar\rho$ is continuous in the duality with continuous bounded functions, for a Borel vector field $\bar u$ satisfying
\[
\int_0^T \intr \abs{\bar u(x,t)}^p \bar\rho(x,t)\,dx\,dt < \infty
\]
for some $p > 1$. Let $\rho \in \mc([0,T]; \calP_p(\R^d))$ be a solution of the following continuity equation
\[
\partial_t \rho + \nabla \cdot (\rho u) = 0
\]
with the velocity fields $u \in L^\infty((0,T); \dot{W}^{1,\infty}(\R^d))$. Then there exists a $C_{u,T} > 0$ depending only on $T$ and $\norm{\nabla u}_{L^\infty}$ such that for all $t \in [0,T]$
\[
W_1^2(\rho, \bar\rho) \leq C_{u,T} \left(W_1^2(\rho_0, \bar\rho_0) + \int_0^t \intr \rho^\e \abs{u^\e - u}^2\,dx\,ds \right),
\]
where $\rho^\varepsilon$ and $u^\varepsilon$ are defined in Remark \ref{rmk:macr_vel}.
\end{lemma}

\begin{rmk}
Since 
\[
\rho^\e_i \abs{u^\e_i - u_i}^2 \leq \intr f^\e_i \abs{ v - u_i}^2\,dv,
\]
Lemma \ref{lem_dbl} particularly implies
\[
W_1^2(\rho_i, \bar\rho_i) \leq C_{u,T} \bigg(W_1^2(\rho_{i0}, \bar\rho_{i0}) + \int_0^t \intdd f^\e_i \abs{v - u_i}^2\,dx\,dv\,ds \bigg),
\]
for $\fori$.
\end{rmk}

We are now in a position to provide the details of proof for Theorem \ref{thm_ktoc}.

\begin{proof}[Proof of Theorem \ref{thm_ktoc}]
We first rewrite the system \eqref{eq:singular_lim} as
\[
\begin{aligned}
&\partial_t \rho_i + \nabla \cdot (\rho_i u_i) = 0, \cr
&\e\partial_t u_i +\e u_i \cdot \nabla u_i =  - u_i -\sum_{j=1}^N \nabla K_{ij}\ast\rho_j + \e \mathrm{e}_i,
\end{aligned}
\]
where $\mathrm{e}_i\coloneqq \partial_t u_i + u_i \cdot \nabla u_i$, for $i=1,\dots,N$.
For the error estimates, we consider the modulated kinetic and interaction energies:
\[
\calE_{K}(f^\e_i| \rho_i, u_i) \coloneqq \frac12\intdd \abs{u_i-v}^2 f^\e_i\,dx\,dv + \frac1{2\e}\intr (\rho_i - \rho^\e_i )K_{ii} * (\rho_i - \rho^\e_i)\,dx.
\]
Straightforward computation yields that for each $i=1,\dots,N$
\[
\begin{aligned}
&\frac12\frac{d}{dt}\intdd \abs{u_i-v}^2 f^\e_i\,dx\,dv + \frac{1}{\e} \intdd  \abs{u_i-v}^2 f^\e_i\,dx\,dv\cr
&\quad = \intdd (u_i - v) \otimes (v-u_i) : \nabla_x u_i  f^\e_i\,dx\,dv - \intdd (v- u_i) \cdot \mathrm{e}_i f\i^\e\,dx\,dv\cr
&\qquad \quad + \frac1\e \intdd (v-u_i) \cdot \bigg(\sum_{j=1}^N \nabla K_{ij} *(\rho_j - \rho^\e_j)  \bigg)f^\e_i\,dx\,dv\cr
&\quad \eqqcolon I + II + III,
\end{aligned}
\]
where
\[
I \leq \norm{\nabla u_i}_{L^\infty}\intdd \abs{u_i-v}^2 f^\e_i\,dx\,dv, 
\]
and
\[
II \leq 4\e\norm{\mathrm{e}_i}_{L^\infty} + \frac1\e\intdd \abs{u_i-v}^2 f^\e_i\,dx\,dv.
\]
For $III$, we use $\nabla K_{ij} \in W^{1,\infty}$ for $i,j=1,\dots,N$ with $i \neq j$ to obtain
\[
\begin{aligned}
III \leq & \frac1\e \intr \rho^\e_i (u^\e_i - u_i)  \cdot \nabla K_{ii} *(\rho_i - \rho^\e_i) \,dx  \cr
&\quad + \frac{1}{\e}\bigg( \intdd \abs{u_i-v}^2 f^\e_i\,dx\,dv \bigg)^{1/2}\sum_{j\neq i}\norm{\nabla K_{ij}}_{W^{1,\infty}} W_1(\rho_j,\rho_j^\e)\cr
\leq & \frac1\e \intr \rho^\e_i (u^\e_i - u_i)  \cdot \nabla K_{ii} *(\rho_i - \rho^\e_i) \,dx + \frac{c_K}{2\e} \intdd \abs{u_i-v}^2 f^\e_i\,dx\,dv \cr
&\quad + \frac{c_K}{2\e} \sum_{j\neq i}W_1^2(\rho_j,\rho_j^\e),
\end{aligned}
\]
where 
\[
c_K \coloneqq \max_{i=1,\dots,N}\sum_{j\neq i}\norm{\nabla K_{ij}}_{W^{1,\infty}}.
\]
We then apply Theorem \ref{thm_mod} and Lemma \ref{lem_dbl} to deduce
\[
\begin{aligned}
&\calE_{K}(f^\e_i|  \rho_i, u_i)  + \frac1\e\bigg(1 -  \max_{i=1,\dots,N}\norm{\nabla u_i}_{L^\infty} \e - 1 - \frac{c_K}{2}  \bigg)\int_0^t \intdd  \abs{u_i-v}^2 f^\e_i\,dx\,dv\,ds  \cr
&\quad \leq  \calE_K(f^\e_{i0}| \rho_{i0}, u_{i0}) + \frac{Cc_K}{2\e} \sum_{j\neq i}W_1^2(\rho_{j0},\rho^\e_{j0}) + C\e \cr
&\qquad + \frac C{\e}\sum_{i=1}^N\int_0^t\intr (\rho_i - \rho^\e_i )K_{ii} * (\rho_i - \rho^\e_i)\,dx\,ds + \frac{Cc_K}{2\e}\sum_{j\neq i}\int_0^t \intdd \abs{u_j - v}^2 f^\e_j\,dx\,dv\,ds,
\end{aligned}
\]
where $C>0$ depends only on $u_1,\dots,u_N$, and $T$, but independent of $\e>0$. 
We now sum over $i=1,\dots,N$, apply Gr\"onwall's lemma to have
\[
\begin{aligned}
&\sum_{i=1}^N\calE_{K}(f^\e_i| \rho_i, u_i) + \frac1\e\sum_{i=1}^N\int_0^t \intdd  \abs{u_i-v}^2 f^\e_i\,dx\,dv\,ds \cr
&\quad   \leq c_0\sum_{i=1}^N\calE_{K}(f^\e_{i0}| \rho_{i0}, u_{i0}) + \frac{c_0}{\e}\sum_{i=1}^N W_1^2(\rho_{i0},\rho^\e_{i0}),
\end{aligned}
\]
where $c_0 > 0$ is independent of $\e > 0$. We finally use \eqref{ini_cond_sing0}, \eqref{ini_cond_sing1}, and \cite[Lemma 4.2]{CJpre} to conclude our desired result.
\end{proof}

\section*{Acknowledgments}

The research of YPC is supported by NRF grant no. 2022R1A2C1002820.
The research of SF is  supported by the Ministry of University and Research (MIUR), Italy under the grant PRIN 2020- Project N. 20204NT8W4, Nonlinear Evolutions PDEs, fluid dynamics and transport equations: theoretical foundations and applications. VI and SF are supported by the ``MMEAN-FIELDSS'' INdAM project N.E53C22001930001 , and by the InterMaths Network, \url{www.intermaths.eu}.

\bibliographystyle{plain}

\end{document}